%% file: Paper_vesselArXiv.tex
\documentclass[smallextended]{svjour3} 

\usepackage{pdfsync}
\usepackage{amsmath,amsfonts}
\usepackage{bm}
\usepackage{graphicx,subfigure}
\usepackage[textsize=footnotesize]{todonotes}
\usepackage{subfigure}
\usepackage[colorlinks]{hyperref}
\usepackage{color} 
\usepackage{amsthm}
\usepackage{appendix}
\usepackage[format=hang, justification=justified]{caption}
\newcommand{\vectornorm}[1]{\left|\left|#1\right|\right|}

\theoremstyle{break}
\newtheorem{Lemma}{Lemma}[section]

\newtheorem{Proposition}[Lemma]{Proposition}
\newtheorem{defn}{Definition}[section]
\newtheorem{Remark}{Remark}[section]
\makeatletter

\begin{document}

\title{Passivity-preserving splitting methods for rigid body systems}
\author{Elena Celledoni\corrauth, Eirik Hoel H{\o}iseth\ and Nataliya Ramzina}
\author{Elena Celledoni, \and
Eirik Hoel H{\o}iseth \and 
Nataliya Ramzina}

\institute{ \at
              Department of Mathematical Sciences, NTNU, 7491 Trondheim, Norway \\
              \email{elena.celledoni@ntnu.no} 
              \email{eirik.hoiseth@math.ntnu.no}           
}


\maketitle

\begin{abstract}
A rigid body model for the dynamics of a marine vessel, used in simulations of offshore pipe-lay operations, gives rise to a set of ordinary differential equations with controls. 
The system is input-output passive. We propose passivity-preserving splitting methods for the numerical solution of a class of problems which includes this system as a special case. 
We prove the passivity-preservation property for the splitting methods, and we investigate stability and energy behaviour in numerical experiments. Implementation is discussed in detail for a special case where the splitting gives rise to the subsequent integration of two completely integrable flows. The equations for the attitude are reformulated on $SO(3)$ using rotation matrices rather than local parametrizations with Euler angles. 
\end{abstract}




\input{Chapters/1_Introduction}

\input{Chapters/Section2}
\input{Chapters/3_Numerical_solutions}
\input{Chapters/5_Numerical_results}

\input{Chapters/6_Concluding_remarks}

\input{Chapters/Appendix}

\clearpage
This work has received funding from the European Unions Horizon 2020
research and innovation programme under the Marie Sklodowska-Curie grant
agreement No. 691070, and from The Research Council of Norway.
We are grateful to T. I. Fossen for useful discussions. We are also grateful to Sergio Blanes and Fernando Casas for useful discussions regarding splitting methods, and for providing highly accurate coefficients for the splitting methods of order 4 and 6 used in the numerical experiments. Part of this work was done while visiting Massey University, Palmerston North, New Zealand, and La Trobe University, Melbourne, Australia.

This is a pre-print of an article published in Multibody System Dynamics. The final authenticated version is available online at: https://doi.org/10.1007/s11044-018-9628-5.

\input{Chapters/Bibliography}
\clearpage

\end{document}

%% file: Chapters/1_Introduction.tex
\section{Introduction}

In this paper we propose passivity preserving splitting methods for the control of input-output passive rigid body systems, and in particular for a model of a marine vessel. The preservation of passivity under numerical discretization is not well known  in the literature. We here propose a simple and general technique to achieve passivity-preservation under numerical discretization when using splitting methods. This technique is applicable to a large class of input-output passive systems, of which the vessel model is a special example.

The control of vessel rigid body equations is important in the simulation of a number of offshore marine operations \cite{marinecontrolsystems}. One example is the simulation of pipeline installations sketched in Figure~\ref{fig2}. 
\begin{figure}
\centering
\includegraphics[width=11cm]{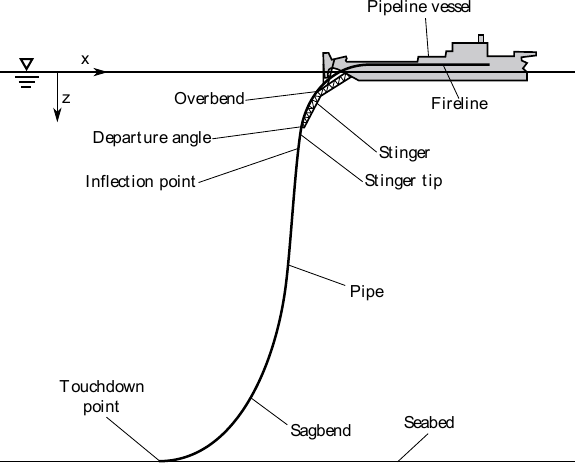}
\caption{The offshore pipe-lay process. The purpose build pipeline vessel uses heavy tension equipment to clamp the pipe on to it. The pipe is then extended in a production line. The two main pipelay methods are the dominant S-lay (shown) and J-lay. In the S-lay method, the pipe is extended horizontally. The name S-lay comes from the S-shape of the pipe from vessel to seabed. A submerged supporting structure, called a stinger, controls curvature and ovalization in the upper part (overbend). Pipe tension controls the curvature in the lower curve (sagbend). The strain must be checked to stay within limits for buckling and ovalization. See \cite{jensen09mac} and the references therein for more details.\label{fig2}}
\end{figure}
The pipe-laying process comprises the modeling of two interacting structures, a vessel and a pipe. It is usual to model the vessel as a torqued rigid body with control, and the long and thin pipe as a rod. 
The parameters to control are the vessel position and velocity, the pay-out speed, and the pipe tension. The control objectives are to determine the position of the touchdown point of the pipe, and to avoid critical deformations and structural failures  \cite{jensen09mac}, \cite{jensen10anp}. The design of damping forces and controls must ensure that the resulting system is input-output passive, i.e. an energy conservation/dissipation property must be satisfied.  

The full pipe-lay problem has already been simulated in \cite{jensen10anp}, using local parametrisations based on Euler angles for both the pipe and the vessel model and with standard numerical techniques. The focus in this paper is 
the the passivity preserving integration of the vessel rigid-body model. 

We consider the six degrees of freedom model for the dynamics of a
marine vessel given in  \cite{jensen10anp} (see also \cite{marinecontrolsystems} and \cite{perez07kmf}), which expressed in matrix-vector form
is
\begin{align}
\dot{\mbox{\boldmath $\eta$}}&=J(\mbox{\boldmath $\eta$})\mbox{\boldmath $\nu$}, \nonumber \\
M\dot{\mbox{\boldmath $\nu$}} + C(\mbox{\boldmath $\nu$})\mbox{\boldmath $\nu$}
 + D(\mbox{\boldmath $\nu$})\mbox{\boldmath $\nu$} + \mathbf{g}(\mbox{\boldmath $\eta$}) &=
\mbox{\boldmath $\tau$} +\mathbf{\mathcal{X}}+ \mathbf{w}, \label{eq:vesseleq}
\end{align}

\noindent where $\mbox{\boldmath $\eta$}$ and $\mbox{\boldmath $\nu$}$ are  generalised position and velocity respectively, $M$ is the system inertia matrix, $C(\mbox{\boldmath $\nu$})$ the skew-symmetric Coriolis-centripetal matrix, $D(\mbox{\boldmath $\nu$})$ the symmetric damping matrix, $\mathbf{g}(\mbox{\boldmath $\eta$})$ the vector of gravitational and buoyancy forces and moments, $\mbox{\boldmath $\tau$}$ the vector of control inputs, $\mathbf{\mathcal{X}}$ the forces and moments from the pipe, and $\mathbf{w}$ environmental disturbances such as wind, waves and currents,  see section~\ref{model} and \ref{controls} for details.  The position variables $\mbox{\boldmath $\eta$}$ include the three Euler angles representing the attitude of the body and evolving on the Lie group $SO(3)$. We reformulate the equations using rotation matrices, and provide numerical approximations that preserve the structure of this manifold.

This vessel model \eqref{eq:vesseleq} is a special example of an input-output passive systems, and in particular an input-state-output port-Hamiltonian system, \cite{VanderSchaft06pHs}. The passivity of the vessel system follows directly from the passivity of port-Hamiltonian systems (see Section~\ref{passivity}).
Our proposed numerical integration methods are splitting methods, a class of integration methods very well known in the geometric numerical integration literature \cite{gni06hairer}. 
We propose to split systems of the type \eqref{eq:vesseleq} into a conservative part, and a part comprising torque, dissipative forces and controls. 
We prove passivity of the proposed splitting methods for the general class of input-state-output port-Hamiltonian systems.
For a simplified model, where  
both control and damping are linear with constant coefficients and the mass matrix has a special structure, the two individual flows are completely integrable, and we discuss the implementation details of their exact solution.

In the presence of PID (proportional-integral-derivative) controls, $\mbox{\boldmath $\tau$}$ depends on $\mbox{\boldmath $\nu$} $, on $\mbox{\boldmath $\eta$}$ and on the integral of $\mbox{\boldmath $\eta$}$ over time. 
We include this integral as an extra unknown to the system to avoid explicit time dependence in the system of equations. 


Finally, we make a numerical study of the order and 
the energy behaviour of the methods. The results are compared with those given by explicit Runge-Kutta methods. 

The outline of the paper is as follows: 
in Section 2 we remark on the passivity of this system, describe the details of the splitting and composition techniques and show their input-output passivity; in Section 3 we describe the vessel equations and describe the implementation details; in Section 4 we report our numerical results illustrating the performance of the methods; Section 5 is devoted to conclusions.

%% file: Chapters/Section2.tex
\section{Energy conservation and passivity } 
\label{passivity}

In the present section we consider the problem of conservation of the total energy. 
In terms of control dynamics, the systems of interest are frequently considered as nearly isolated, and the energy conservation law is formulated in terms of the \textit{passivity} of the system.  We will use 
the following definition of passivity.
\begin{defn}
\label{defpass} \cite{egeland} (Chapter 2)
\newline
Consider a model 
\begin{align*}
\dot {\mathbf{y}}  &= \mathbf{f}(\mathbf{y}, \, \mathbf{u}), \\
\boldsymbol{\zeta} &=\mathbf{h}(\mathbf{y}).
\end{align*}
Suppose that there is a \emph{storage function} $V(\mathbf{y})\geq 0$ and a dissipation function $g(\mathbf{y}) \geq 0$ so that the time derivative of $V$ for a solution of the system satisfies 
\begin{equation*}
\dot{V}=\nabla V(\mathbf{y})^T\mathbf{f}(\mathbf{y}, \, \mathbf{u})= \mathbf{u}^T \boldsymbol{\zeta}-g(\mathbf{y}),
\end{equation*} 
 for all control inputs $\mathbf{u}$. Then the system with input $\mathbf{u}$ and output $\boldsymbol{\zeta}$ is said to be (input-output) passive. 
\end{defn}

Notice that the passivity condition is equivalent to
\begin{equation}
\label{passivity1}
V(\mathbf{y}(t))-V(\mathbf{y}(0))=\int_0^t\left(\mathbf{u}(s)^T \boldsymbol{\zeta}(s)-g(\mathbf{y}(s))\right)\,ds.
\end{equation}

Let $H:\mathbf{R}^m\rightarrow\mathbf{R}$ be a total energy function, $\mbox{\boldmath $\xi$}(t)\in\mathbf{R}^m$, $\mbox{\boldmath $\tau$}(t)\in\mathbf{R}^p$, $S(\mbox{\boldmath $\xi$})\in \mathbf{R}^{m\times m}$ skew-symmetric,  $G(\mbox{\boldmath $\xi$})\in \mathbf{R}^{m\times p}$, $D(\mbox{\boldmath $\xi$})\in \mathbf{R}^{p\times p}$ positive definite.
Consider the system 
\begin{eqnarray}
\label{eq:xiDiff}
\dot{\mbox{\boldmath $\xi$}} &=& S(\mbox{\boldmath $\xi$})\nabla H(\mbox{\boldmath $\xi$}) -G\left(D(\mbox{\boldmath $\xi$})\mbox{\boldmath $\nu$} - \mbox{\boldmath $\tau$}\right),\\
\mbox{\boldmath $\nu$}&=&G^T\nabla H(\mbox{\boldmath $\xi$})
\end{eqnarray}
also known as an input-state-output port-Hamiltonian system with input $\mbox{\boldmath $\tau$}$ and output $\mbox{\boldmath $\nu$}$, \cite{VanderSchaft06pHs}.
\begin{Proposition}\label{passivityPROP} The system \eqref{eq:xiDiff}  
is input-output passive with input $\mbox{\boldmath $\tau$}$ and output $\mbox{\boldmath $\nu$}$.
\end{Proposition}

\begin{proof}
Differentiating $H$ with respect to time and using \eqref{eq:xiDiff} we find
\begin{equation}
\dot{H} = \nabla H(\mbox{\boldmath $\xi$})^T\dot{\mbox{\boldmath $\xi$}} = \nabla H(\mbox{\boldmath $\xi$})^T S(\mbox{\boldmath $\xi$})\nabla H(\mbox{\boldmath $\xi$}) +  \mbox{\boldmath $\nu$}^T (-D \mbox{\boldmath $\nu$} + \mbox{\boldmath $\tau$}) = \mbox{\boldmath $\nu$}^T (-D \mbox{\boldmath $\nu$} + \mbox{\boldmath $\tau$}),
\label{eq:HamDiff}
\end{equation}
where the last equality in \eqref{eq:HamDiff} follows from the skew symmetry of $S$.

If we choose $V = H$, $\mathbf{y} = \mbox{\boldmath $\xi$}$, $\mathbf{u}=\mbox{\boldmath $\tau$}$ and $\boldsymbol{\zeta}=\mbox{\boldmath $\nu$}$, then it follows from \eqref{eq:HamDiff} that the system \eqref{eq:xiDiff} satisfies Definition \ref{defpass} with  $g(\mathbf{y})=\mbox{\boldmath $\nu$}^T \, D \mbox{\boldmath $\nu$}$ and $\mathbf{h}(\mathbf{y}) =G^T\nabla H(\mbox{\boldmath $\xi$})$. Thus the system is passive. The requirement $g(\mathbf{y}) \geq 0$ is fulfilled because the matrix $D$ is positive-definite. 
\end{proof}

%% file: Chapters/3_Numerical_solutions.tex
\subsection{Splitting methods}
\label{numerical}
Consider a general system of first order ODEs where the right hand side vector field admits the splitting

\begin{equation}
\label{eq:split}
\dot{\mathbf{y}} = \mathbf{S_1}(\mathbf{y}) + \mathbf{S_2}(\mathbf{y}).
\end{equation}

We generate a numerical approximation of \eqref{eq:split} by composing the exact flows of $\mathbf{S_1}$ and $\mathbf{S_2}$, computed on time intervals of suitable length, with suitable initial values. Let $\boldsymbol{\Phi}_{h}^{[\mathbf{S_1}]}$ and $\boldsymbol{\Phi}_{h}^{[\mathbf{S_2}]}$ denote, respectively, the exact flow maps of $\mathbf{S_1}$ and $\mathbf{S_2}$. The Lie-Trotter splitting, giving approximations of order $1$, is
\begin{equation}\label{eq:lt}
\mathbf{y}^{j+1} = \boldsymbol{\Phi}_{h}^{[\mathbf{S_2}]}\circ \boldsymbol{\Phi}_{h}^{[\mathbf{S_1}]}
\left(\mathbf{y}^{j}\right),\quad \mathrm{or}\quad \mathbf{y}^{j+1} =\boldsymbol{\Phi}_{h}^{[\mathbf{S_1}]}\circ \boldsymbol{\Phi}_{h}^{[\mathbf{S_2}]} \,
\left(\mathbf{y}^{j}\right).
\end{equation}
An approximation of order 2 is given by the classical Strang splitting scheme as
\begin{equation}\label{eq:s-v0}
\mathbf{y}^{j+1} = \boldsymbol{\Phi}_{h/2}^{[\mathbf{S_1}]}\circ \boldsymbol{\Phi}_{h}^{[\mathbf{S_2}]} \circ\boldsymbol{\Phi}_{h/2}^{[\mathbf{S_1}]}\,
\left(\mathbf{y}^{j}\right),
\end{equation}
or alternatively
\begin{equation}\label{eq:s-v}
\mathbf{y}^{j+1} = \boldsymbol{\Phi}_{h/2}^{[\mathbf{S_2}]}\circ \boldsymbol{\Phi}_{h}^{[\mathbf{S_1}]} \circ\boldsymbol{\Phi}_{h/2}^{[\mathbf{S_2}]}\,
\left(\mathbf{y}^{j}\right).
\end{equation}
The two flows $\mathbf{S_1}$ and $\mathbf{S_2}$ can also be combined to obtain splitting
methods of higher order. Specifically, symmetric splitting schemes of the
type 
\begin{equation}
\label{splittinggeneral}
\mathbf{y}^{j+1} = \boldsymbol{\Phi}_{a_1h}^{[\mathbf{S_2}]}\circ \boldsymbol{\Phi}_{b_1h}^{[\mathbf{S_1}]} \circ\boldsymbol{\Phi}_{a_2h}^{[\mathbf{S_2}]}
\circ \cdots \circ \boldsymbol{\Phi}_{a_{m+1}h}^{[\mathbf{S_2}]} \circ \cdots \circ \boldsymbol{\Phi}_{b_1h}^{[\mathbf{S_1}]} \circ \boldsymbol{\Phi}_{a_1h}^{[\mathbf{S_2}]}\,
\left(\mathbf{y}^{j}\right),
\end{equation}
are considered. These are a generalization of the scheme \eqref{eq:s-v}.
The Strang splitting can be reproduced from the general formula \eqref{splittinggeneral} by choosing $m=1$, $a_1 = b_1 = 1/2$ and $a_2 = 0$.
The coefficients of a $4$th and a $6$th order scheme are reported in Appendix \ref{SplittingCoef}.

%

We refer
to \cite{blanes} for other schemes which are used in the numerical experiments.
We remark that the exact flows of $\mathbf{S_1}$ and $\mathbf{S_2}$ can be replaced with suitable high order numerical approximations, while keeping the overall order of the splitting method unchanged.

%
%





\subsection{Passivity-preserving splitting  of input-state-output port-Hamiltonian systems}
\label{splittingvessel}

We split the 
equations \eqref{eq:xiDiff}
by splitting the skew symmetric matrix $S(\mbox{\boldmath $\xi$})$ into the sum of two skew-symmetric terms $S(\mbox{\boldmath $\xi$})=S_1(\mbox{\boldmath $\xi$})+S_2(\mbox{\boldmath $\xi$})$.

\begin{Proposition}
\label{propositionpassivitysplitting}
The order $1$  splitting method \eqref{eq:lt}, when applied to the flows of the two systems 
\begin{equation}
\label{splitting}
\left\{\begin{array}{lcl}
\dot{\mbox{\boldmath $\xi$}}&=&S_1(\mbox{\boldmath $\xi$})\nabla H(\mbox{\boldmath $\xi$})\\
\mbox{\boldmath $\xi$}(0)&=&\mbox{\boldmath $\xi$}_0,
\end{array}\right.\, \mathrm{on}\,\, [0,h],\quad
\left\{\begin{array}{lcl}
\dot{\tilde{\mbox{\boldmath $\xi$}}}&=&S_2(\tilde{\mbox{\boldmath $\xi$}})\nabla H(\tilde{\mbox{\boldmath $\xi$}})-G\left(D\tilde{\mbox{\boldmath $\nu$}} - \mbox{\boldmath $\tau$}\right)\\
\tilde{\mbox{\boldmath $\xi$}}(0)&=&\tilde{\mbox{\boldmath $\xi$}}_0,
\end{array}\right.\, \mathrm{on}\,\, [0,h],
\end{equation}
and with $\tilde{\mbox{\boldmath $\xi$}}_0=\mbox{\boldmath $\xi$}(h)$, is input-output passive with input $\mbox{\boldmath $\tau$}$ and output $\tilde{\mbox{\boldmath $\nu$}}=G\,\nabla H(\tilde{\mbox{\boldmath $\xi$}})$.
\end{Proposition}

\begin{proof}
We want to prove that \eqref{passivity1} holds with $V=H$. We have
\begin{eqnarray*}
H(\tilde{\mbox{\boldmath $\xi$}}(h))-H(\mbox{\boldmath $\xi$}_0)&=& H(\tilde{\mbox{\boldmath $\xi$}}(h))-H(\tilde{\mbox{\boldmath $\xi$}}_0)+H(\mbox{\boldmath $\xi$}(h))-H(\mbox{\boldmath $\xi$}_0)\\
&=& \int_0^h\nabla H(\tilde{\mbox{\boldmath $\xi$}}(s))^T\dot{\tilde{\mbox{\boldmath $\xi$}}}(s)\,ds +\int_0^h\nabla H(\mbox{\boldmath $\xi$}(s))^T\dot{\mbox{\boldmath $\xi$}}(s)\,ds \\
&=&\int_0^h\nabla H(\tilde{\mbox{\boldmath $\xi$}}(s))^T\left(S_2(\tilde{\mbox{\boldmath $\xi$}})\nabla H(\tilde{\mbox{\boldmath $\xi$}})-G\left(D\tilde{\mbox{\boldmath $\nu$}}-\mbox{\boldmath $\tau$}\right)\right)\,ds +\\
& &+\int_0^h\nabla H(\mbox{\boldmath $\xi$}(s))^TS_1(\mbox{\boldmath $\xi$}(s))\nabla H(\mbox{\boldmath $\xi$}(s))\,ds, \\
\end{eqnarray*}
and by using the skew-symmetry of $S_1$ and $S_2$ we get finally
\begin{eqnarray*}
H(\tilde{\mbox{\boldmath $\xi$}}(h))-H(\mbox{\boldmath $\xi$}_0)&=&-\int_0^h\nabla H(\tilde{\mbox{\boldmath $\xi$}}(s))^TG\left(D\tilde{\mbox{\boldmath $\nu$}}-\mbox{\boldmath $\tau$}\right)\,ds\\
&=&-\int_0^h \tilde{\mbox{\boldmath $\nu$}}^T\left(D\tilde{\mbox{\boldmath $\nu$}}-\mbox{\boldmath $\tau$}\right)\,ds.
\end{eqnarray*}
Notice that the order of composition of the flows is immaterial.
If we now choose $V=H$, $\mathbf{y}=\tilde{\mbox{\boldmath $\xi$}}$, $\mathbf{u}=\mbox{\boldmath $\tau$}$ and 
$\boldsymbol{\zeta}=\tilde{\mbox{\boldmath $\nu$}}$, then it follows that the composition of the flows of the systems \eqref{splitting} satisfies Definition \ref{defpass} with  $g(\mathbf{y})=\tilde{\mbox{\boldmath $\nu$}}^T \, D \tilde{\mbox{\boldmath $\nu$}}$ and output $\tilde{\mbox{\boldmath $\nu$}}$.
\end{proof}

\begin{Remark} The passivity result proved in Proposition~\ref{propositionpassivitysplitting} for the simple order $1$ splitting of two flows, generalises to
the order $1$ splitting and composition of $m$ flows arising from writing $S$ as a sum of $m$ skew-symmetric terms, $S(\mbox{\boldmath $\xi$})=\sum_{i=1}^mS_i (\mbox{\boldmath $\xi$})$. This  splitting method is a direct generalisation of  \eqref{eq:lt}.
\end{Remark}

\begin{Remark} The passivity result proved in Proposition~\ref{propositionpassivitysplitting} for the simple splitting method of order $1$ generalises easily to the Strang splitting given by \eqref{eq:s-v0} and \eqref{eq:s-v}. Note however, that it is crucial in the proof that the coefficients of the splitting method are real and positive. There are no splitting methods of order higher than $2$ with this property, \cite{blanesComplexcoeff}.  But there exist splitting methods of order greater than or equal to $3$  with complex coefficients which have positive real part, see for instance \cite{blanesComplexcoeff}. The construction and implementation of passivity preserving splitting methods of higher order for rigid body dynamics using complex coefficients is not pursued here, but can be an interesting subject for future investigations. 
\end{Remark}

\begin{Remark}
If the splitting method used for the integration has order $p$, then it is possible to replace the exact flows of the two subsystems with approximate numerical flows of order $p$ without compromising the order of the method. If the corresponding numerical flows are energy-preserving for the first system of the splitting, and energy-dissipative for the second system, then the overall splitting method with positive coefficients remains input-output passive. Methods with such preservation/dissipation properties  
can be found among implicit Runge-Kutta schemes for polynomial vector fields \cite{celledoni09epr}, and among discrete gradient methods and their generalizations for general vector fields \cite{mclachlanDG}, \cite{hairer}.
\end{Remark} 

We next consider a simplified system for which the proposed splitting method always dissipates the total energy independently of the sign of the (real) coefficients of the splitting method.
Specifically, we consider the system
\begin{equation}
\label{system2}
\dot{\mbox{\boldmath $\xi$}} = (S(\mbox{\boldmath $\xi$})-D(\mbox{\boldmath $\xi$}) )\nabla H(\mbox{\boldmath $\xi$}) ,
\end{equation}
with $S\in\mathbf{R}^{m\times m}$ skew symmetric, $D\in\mathbf{R}^{m\times m}$ symmetric and positive definite. Note that the system \eqref{eq:xiDiff} can be rewritten in this form under the assumption that $\mbox{\boldmath $\tau$}=-KG^T\nabla H(\mbox{\boldmath $\xi$})$ with $K\in\mathbf{R}^{p\times p}$ symmetric and positive semi-definite. Let $H$ be a quadratic energy function $ H(\mbox{\boldmath $\xi$})=\mbox{\boldmath $\xi$}^T\tilde{H} \,\mbox{\boldmath $\xi$}$, with $\tilde{H}\in\mathbf{R}^{m\times m}$ symmetric and positive definite. We can then define the energy norm
$$\|\mbox{\boldmath $\xi$}\|_{H}:=\sqrt{\mbox{\boldmath $\xi$}^T\tilde{H} \,\mbox{\boldmath $\xi$}},$$ associated with the inner product $\langle \cdot , \cdot\rangle_H:= \langle \cdot , \tilde{H} \cdot\rangle$. 
\begin{Lemma}\label{lemma1}
System \eqref{system2} satisfies
\begin{equation}
\label{onesidedLipshitzconstant}
\langle (S(\mbox{\boldmath $\xi$})-D(\mbox{\boldmath $\xi$}) )\nabla H(\mbox{\boldmath $\xi$}) , \mbox{\boldmath $\xi$} \rangle_H\le-2\sigma_{\min} \, \|\mbox{\boldmath $\xi$}\|_H^2
\end{equation}
and
\begin{equation}
\label{onesidedLipshitzconstantminus}
\langle -(S(\mbox{\boldmath $\xi$})-D(\mbox{\boldmath $\xi$}) )\nabla H(\mbox{\boldmath $\xi$}) , \mbox{\boldmath $\xi$} \rangle_H\le 2\sigma_{\max} \, \|\mbox{\boldmath $\xi$}\|_H^2
\end{equation}

with $\sigma_{\max}\ge\sigma_{\min}> 0$.
\end{Lemma}
\begin{proof}
We have $\nabla H =2\,\tilde{H} \, \mbox{\boldmath $\xi$}$, so
\begin{equation}
\label{eq:innprod}
\langle (S(\mbox{\boldmath $\xi$})-D(\mbox{\boldmath $\xi$}) )\nabla H(\mbox{\boldmath $\xi$}) , \mbox{\boldmath $\xi$} \rangle_H = -2\mbox{\boldmath $\xi$}^T\tilde{H}(S(\mbox{\boldmath $\xi$})+D(\mbox{\boldmath $\xi$}) )\tilde{H} \, \mbox{\boldmath $\xi$}
=-2\mbox{\boldmath $\xi$}^T\tilde{H}D\tilde{H}\mbox{\boldmath $\xi$},
\end{equation}
If $\sigma_{\max} > 0$ and $\sigma_{\min} > 0$ are the maximum and minimum eigenvalue of $\tilde{H}^{\frac{1}{2}}D\tilde{H}^{\frac{1}{2}}$, we have
\begin{equation}
\label{eq:eigenv}
\sigma_{\min} \|\mbox{\boldmath $\xi$}\|^2_H=\sigma_{\min}\|\tilde{H}^{\frac{1}{2}} \mbox{\boldmath $\xi$}\|_2^2\le 
\mbox{\boldmath $\xi$}^T\tilde{H}D\tilde{H}\mbox{\boldmath $\xi$}
\le \sigma_{\max} 
\|\tilde{H}^{\frac{1}{2}} \mbox{\boldmath $\xi$}\|^2_2=\sigma_{\max} \|\mbox{\boldmath $\xi$}\|^2_H.
\end{equation}

Notice that, by Sylvester's law of inertia,  $\tilde{H}^{\frac{1}{2}}D\tilde{H}^{\frac{1}{2}}$ is symmetric positive definite since $D$ and $\tilde{H}$ are.  From \eqref{eq:innprod} and \eqref{eq:eigenv}, \eqref{onesidedLipshitzconstant} and \eqref{onesidedLipshitzconstantminus} follow.
\end{proof}
The following Lemma is an adaptation of the results of \cite[p.180]{hairer91sod}.
\begin{Lemma} \label{lemma2} 
Assume $\sigma_{\max}\ge\sigma_{\min} > 0$ and \eqref{onesidedLipshitzconstant} and \eqref{onesidedLipshitzconstantminus} hold. 
Then for $t\ge 0$ along continuous solutions $\mbox{\boldmath $\xi$}$ of \eqref{system2} we have
\begin{equation}
\label{dissipationH}
H(\mbox{\boldmath $\xi$}(t))\le e^{-4t\sigma_{\min}} H(\mbox{\boldmath $\xi$}(0));
\end{equation}
for $t \le 0$ along continuous solutions $\mbox{\boldmath $\xi$}$ of \eqref{system2} we have
\begin{equation}
\label{propagationH}
H(\mbox{\boldmath $\xi$}(t))\le e^{-4t\sigma_{\max}} H(\mbox{\boldmath $\xi$}(0)).
\end{equation}
\end{Lemma}
\begin{proof}
We have
$$\frac{d H(\mbox{\boldmath $\xi$}(t))}{dt}=\nabla H(\mbox{\boldmath $\xi$})^T\dot{\mbox{\boldmath $\xi$}}(t)=2\langle (S(\mbox{\boldmath $\xi$})-D(\mbox{\boldmath $\xi$}) )\nabla H(\mbox{\boldmath $\xi$}) , \mbox{\boldmath $\xi$} \rangle_H,$$
for $t\ge 0$.
From \eqref{onesidedLipshitzconstant} 
we then get
$$\frac{d H(\mbox{\boldmath $\xi$}(t))}{dt}\le-4\sigma_{\min} H(\mbox{\boldmath $\xi$}(t)),
$$
$H(t)=H(\mbox{\boldmath $\xi$}(t))$ is a continuous function of $t$ if $\mbox{\boldmath $\xi$}(t)$ is. Integrating the obtained differential inequality we get \eqref{dissipationH}.

Notice that integrating \eqref{system2} with negative time corresponds to solving 
$$\dot{\mbox{\boldmath $\xi$}} = -(S(\mbox{\boldmath $\xi$})-D(\mbox{\boldmath $\xi$}) )
\nabla H(\mbox{\boldmath $\xi$} ) $$ for $t\ge 0$, so differentiating $H(\mbox{\boldmath $\xi$}(t))$ with respect to $t$ and using \eqref{onesidedLipshitzconstantminus}, we arrive in this case \eqref{propagationH}. 
\end{proof}

\vskip0.2cm

\vskip0.2cm
Assume $S=S_1+S_2$ with both $S_1$ and $S_2$ skew-symmetric. In what follows we give a sufficient condition for \eqref{splittinggeneral} to be energy dissipative even when the coefficients of the splitting method are real and both positive and negative.
\begin{Proposition}
\label{propositionpassivitysplittingquadraticenergy}
The splitting method \eqref{splittinggeneral} of order $p\ge 1$ given by composing the flows  
{\small
\begin{equation}
\mathbf{S}_1: \left\{\begin{array}{lcl}
\dot{\mbox{\boldmath $\xi$}}&=&S_1(\mbox{\boldmath $\xi$})\nabla H(\mbox{\boldmath $\xi$})\\
\mbox{\boldmath $\xi$}(0)&=&\mbox{\boldmath $\xi$}_0
\end{array}\right.\, \mathrm{on}\,\, [0,h],\quad
\mathbf{S}_2: \left\{\begin{array}{lcl}
\dot{\tilde{\mbox{\boldmath $\xi$}}}&=&(S_2(\tilde{\mbox{\boldmath $\xi$}})  - D ( \tilde{\mbox{\boldmath $\xi$}}) )\nabla H(\tilde{\mbox{\boldmath $\xi$}})\\
\tilde{\mbox{\boldmath $\xi$}}(0)&=&\tilde{\mbox{\boldmath $\xi$}}_0
\end{array}\right.\, \mathrm{on}\,\, [0,h],
\end{equation}
}
and with $\tilde{\mbox{\boldmath $\xi$}}_0=\mbox{\boldmath $\xi$}(h)$, is energy dissipative if
\begin{equation}
\label{eq:condition}
2\,\sum_{i=1}^m\mu_ia_i+\mu_{m+1}a_{m+1}\ge 0, \quad \mu_i:=\left\{\begin{array}{ccc}
\sigma_{\min} &  \mathrm{if} & a_i>0,\\
\sigma_{\max} & \mathrm{if} & a_i <0,
\end{array}\right.
\end{equation}
with $\sigma_{\max}$ and $\sigma_{\min}$ the maximum and minimum eigenvalue of $\tilde{H}^{\frac{1}{2}}D\tilde{H}^{\frac{1}{2}}$.
\end{Proposition}
\begin{proof}
By the order $1$ conditions of the splitting method \cite[sec. III.3.4]{gni06hairer} we have
$$2\,\sum_{i=1}^ma_m+a_{m+1}=1,\quad    2\sum_{i=1}^mb_i=1.$$
Consider 
$$\mu_i:=\left\{\begin{array}{ccc}
\sigma_{\min} &  \mathrm{if} & a_i>0,\\
\sigma_{\max} & \mathrm{if} & a_i <0.
\end{array}\right.
$$
For the solution, $\tilde{\mbox{\boldmath $\xi$}}(t)$, of the system $\mathbf{S}_2$, we have from Lemma~\ref{lemma2} that on the time interval $[0,a_ih]$, 
$$H(\tilde{\mbox{\boldmath $\xi$}}(a_ih))
\le 
e^{-4\mu_ia_ih}\,H(\tilde{\mbox{\boldmath $\xi$}}(0)).$$
For the solution $\mbox{\boldmath $\xi$}(t)$ of  $\mathbf{S}_1$ on $[0,b_ih]$, we have
$$H(\mbox{\boldmath $\xi$}(b_ih))=
H(\mbox{\boldmath $\xi$}( 0)).$$
Using \eqref{splittinggeneral} and the above inequalities we have
$$H(\mbox{\boldmath $\xi$}^{j+1})\le e^{-4\mu_1 a_1h}e^{-4\mu_2 a_2h}\cdots e^{-4\mu_{m+1} a_{m+1}h}\cdots e^{-4\mu_2 a_2h} e^{-4\mu_1 a_1h}\,   H(\mbox{\boldmath $\xi$}^{j}), 
$$
and so
$$H(\mbox{\boldmath $\xi$}^{j+1})\le e^{- 4(2\sum_{i=1}^m\mu_ia_i+\mu_{m+1}a_{m+1})}H(\mbox{\boldmath $\xi$}^{j})\le H(\mbox{\boldmath $\xi$}^{j}),$$
if \eqref{eq:condition} is satisfied. 
\end{proof}

The coefficients of the method of order four used in our numerical experiments are given in Appendix~\ref{SplittingCoef}. For this method we obtain that condition \eqref{eq:condition} is satisfied  when 
$\frac{\sigma_{\max}}{\sigma_{\min}}\le \frac{2(a_1+a_2)+a_4}{-2a_3}\approx 12$; and interchanging the role of $\mathbf{S}_1$ and $\mathbf{S}_2$, we obtain instead the condition  $\frac{\sigma_{\max}}{\sigma_{\min}}\le \frac{(b_1+b_2)}{-b_3}\approx 4.5$. 

\input{Chapters/2_Vessel_model}

\subsection{Implementation}
\label{implementation}

In the numerical experiments, and in this section, we will consider the case when the mass matrix has the form  \eqref{MRB}. This occurs for example  when the added mass is zero, i.e. $M_A=O$, or when $M_A$ has the same diagonal structure as $M_{RB}$. We will also assume the damping matrix $D$ to be constant. 
As a consequence the term $\mathbf{p}\times \mathbf{v}$ in the equation for the angular momentum is zero (because $\mathbf{p}$ and $\mathbf{v}$ are parallel). For this reason the matrix $S$ can be modified to have a skew symmetric diagonal block $-\widehat{\mbox{\boldmath $\omega$}}$ in the upper left diagonal corner, while both off diagonal blocks of the type $\hat{\mathbf{p}}$ disappear.
The equations \eqref{eq:vessel} simplify and become
\begin{equation}\label{eq:vflows}
\begin{aligned}
\dot{\mathbf{p}} &=-\widehat{\mbox{\boldmath $\omega$}}\,\mathbf{p} -\left(D_t\mathbf{v}+Q^T\mathcal{A}\mathbf{x}-\mbox{\boldmath $\tau$}_t\right), \\
\dot{\mathbf{m}} &= -\widehat{\mbox{\boldmath $\omega$}} \,\mathbf{m} - \left(D_r\mbox{\boldmath $\omega$}+\widehat{\mbox{\boldmath $\mu$}}\,\mathbf{q}_3-\mbox{\boldmath $\tau$}_r\right), \\
\dot{Q}&=Q\,\widehat{\mbox{\boldmath $\omega$}},\\
\dot{\mbox{\boldmath $\mu$}}&= -\mathcal{G}\,\widehat{\mbox{\boldmath $\omega$}}\, \mathbf{q}_3,\\
\dot{\mathbf{x}} &= Q\,\mathbf{v},\\
\dot{\mbox{\boldmath $\varphi$}}_{\boldsymbol{\theta}}&=\tilde{\mbox{\boldmath $\theta$}}, \\
\dot{\mbox{\boldmath $\varphi$}}_{\mathbf{x}}&=\tilde{\mathbf{x}}.
\end{aligned} 
\end{equation}
This case is particularly interesting because it is possible to split $S$ as the sum of two skew symmetric terms $S=S_1+S_2$ while obtaining two completely integrable flows for which explicit solutions can be found. The system $\mathbf{S_1}$, arising from $S_1$, is a free-rigid body flow, as in the more general case of the previous section. Conversely, we will see that the system $\mathbf{S_2}$, arising from $S_2$, is linear and can be integrated exactly using the variation of constants formula. Proposition~\ref{propositionpassivitysplitting} applies directly to this case. In the more general case discrete gradient methods and their higher order generalisations, see  \cite{mclachlanDG} and \cite{hairer}, can be applied to solve $\mathbf{S_2}$.


In what follows, regarding notation, we reparametrize time using $\gamma$, such that $\gamma = 0$ represents the current starting point. We also use the convention that index $0$ for a variable, e.g. $\mathbf{x}_0$, denotes the current initial value (at $\gamma = 0$).
\subsubsection{System $\mathbf{S_1}$: Free-rigid body integration}

The system $\mathbf{S_1}$  amounts to a system of free rigid body equations
\begin{equation}\label{eq:vflows1}
\mathbf{S_1} = \left\{\begin{aligned}
\dot{\mathbf{p}} &=-\widehat{\mbox{\boldmath $\omega$}}\,\mathbf{p}, \\
\dot{\mathbf{m}} &= \mathbf{m} \times \mbox{\boldmath $\omega$}, \\
\dot{Q}&=Q\,\widehat{\mbox{\boldmath $\omega$}},\\
\dot{\mbox{\boldmath $\mu$}}&=\mathbf{0},\\
\dot{\mathbf{x}} &= \mathbf{0}, \\ 
\dot{\mbox{\boldmath $\varphi$}}_{\boldsymbol{\theta}}&=\mathbf{0}, \\
\dot{\mbox{\boldmath $\varphi$}}_{\mathbf{x}}&=\mathbf{0}.
\end{aligned} \right. 
\end{equation}
This system is completely integrable. To integrate these equations we employ techniques developed in \cite{celledoni06ets} and \cite{celledoni08tec}. 


$\mathbf{S_1}$ is integrated several times on some interval  $[0,\alpha h]$ as a part of the splitting methods, with $h$ being the step size of integration and $\alpha=a_i$, $i=1,\dots , m+1$, i.e. one of the coefficients of the splitting method. We observe that $\dot{\mbox{\boldmath $\omega$}}$ and $\dot{Q}$  do not depend on $\mathbf{v}$ and $\mathbf{x}$. As $\mbox{\boldmath $\omega$}$ can also be computed independently of $Q$, we proceed as follows. We first compute $\mbox{\boldmath $\omega$}$ using explicit formulae based on Jacobi elliptic functions, see \cite{celledoni08tec} Proposition 2.1. Then using $\mbox{\boldmath $\omega$}$, we solve numerically the equation for $Q$ in \eqref{eq:vflows1} using the Magnus series expansion. For a method of order $2$ the computed approximation for $Q$ is
$$
Q^{[2]}=Q_0\exp\left(h\, \widehat{ \mbox{\boldmath $\omega$}^{(\frac{1}{2})} }\right),
$$
with $\mbox{\boldmath $\omega$}^{(\frac{1}{2})}=\mbox{\boldmath $\omega$}(\frac{\alpha}{2h})$.
The two approximations of order $4$ and $6$ require the computation of commutators, and the use of quadrature. They give the approximations:
{\small
$$
Q^{[4]}=Q_0\exp\left(\alpha_1-\frac{1}{12}[\alpha_1,\alpha_2]\right),\quad \alpha_1:=\frac{h}{2}(\widehat{\mbox{\boldmath $\omega$}(c_1h)}+\widehat{\mbox{\boldmath $\omega$}(c_2h)}),\quad \alpha_2:=\sqrt{3}h(\widehat{\mbox{\boldmath $\omega$}(c_2h)}-\widehat{\mbox{\boldmath $\omega$}(c_1h)}),
$$
}
with $c_1=\frac{1}{2}-\frac{\sqrt{3}}{6}$ $c_2=\frac{1}{2}+\frac{\sqrt{3}}{6}$; and
$$
Q^{[6]}=Q_0\exp\left(\alpha_1+\frac{1}{12}\alpha_3+\frac{1}{240}[-20\alpha_1-\alpha_3+s_1,\alpha_2+r_1]\right),$$
where
$$s_1:=[\alpha_1,\alpha_2],\quad r_1=-\frac{1}{60}[\alpha_1,2\alpha_3+s_1],$$
and
{\small
$$ \alpha_1:=h\widehat{\mbox{\boldmath $\omega$}(c_2h)},\quad \alpha_2:= \frac{\sqrt{15}h}{3}(\widehat{\mbox{\boldmath $\omega$}(c_3h)}-\widehat{\mbox{\boldmath $\omega$}(c_1h)}),\quad \alpha_3:=\frac{10h}{3}(\widehat{\mbox{\boldmath $\omega$}(c_3h)}-2\widehat{\mbox{\boldmath $\omega$}(c_2h)})+\widehat{\mbox{\boldmath $\omega$}(c_1h)}),
$$
}
with $c_1=\frac{1}{2}-\frac{\sqrt{15}}{10}$, $c_2=\frac{1}{2}$ and $c_3=\frac{1}{2}+\frac{\sqrt{15}}{10}$.
See also \cite{celledoni06ets}, \cite{gustavson} for details. This method will be referred to as the Magnus method. It was shown in  \cite{blanes02} that this method requires a minimal number of commutators. One can truncate the Magnus expansion to arbitrary high order. Here we have implemented this method to order 2, 4, and 6, and combined it with splittings of the same order.


Once $\mbox{\boldmath $\omega$}$ and $Q$ are computed to the desired accuracy, $\mathbf{p}$ can be easily obtained by
$$
\begin{array}{rcl}
\mathbf{p}(\gamma) &=& Q^TQ_0\mathbf{p}_0, \\
\end{array}
$$
while
\begin{eqnarray*}
\mbox{\boldmath $\mu$}(\gamma)&=& \mbox{\boldmath $\mu$}_0,\\
\mathbf{x}(\gamma)&=&\mathbf{x}_0,\\
\mbox{\boldmath $\varphi$}_q(\gamma)&=& \mbox{\boldmath $\varphi$}_{q,0},\\ 
\mbox{\boldmath $\varphi$}_x(\gamma)&=&\mbox{\boldmath $\varphi$}_{x,0}.
\end{eqnarray*}


The rotational part of the equations evolves on $SO(3)$. The Magnus method respects the manifold structure of this Lie gorup. Compared to using Euler angles, this approach avoids singularities by construction. 
This may not be important when simulating the vessel alone, because only a certain limited range of rotations are allowed. However, it becomes more difficult to avoid the singularities of Euler angles  when considering the attitude of the vessel and of each cross sections of the pipe simultaneously,  \cite{jensen10anp}.
The Magnus expansion (which by construction provides an attitude that evolves on $SO(3)$) can also be applied when the mass matrix of the system is more general and the added mass matrix is not diagonal.

\subsubsection{System $\mathbf{S_2}$: Integration of the damping, gravitational and control forces.}

It is easy to realize that the system of differential equations $\mathbf{S_2}$ 
\begin{equation}
\label{eq:vflows2}
\mathbf{S_2} = \left\{\begin{aligned}
\dot{\mathbf{p}} &=  -\left(D_tm_v^{-1}\mathbf{p}+Q^T\mathcal{A}\mathbf{x}-\mbox{\boldmath $\tau$}_t\right),\\
\dot{\mathbf{m}} &= - \left(D_rT^{-1}\mathbf{m}+\widehat{\mbox{\boldmath $\mu$}}\,\mathbf{q}_3-\mbox{\boldmath $\tau$}_r\right),\\
\dot{Q}&= \mathbf{0},\\
\dot{\mbox{\boldmath $\mu$}}&= - \mathcal{G}\,\widehat{\mbox{\boldmath $\omega$}}\, \mathbf{q}_3,\\
\dot{\mathbf{x}} &= Q\mathbf{v}, \\
\dot{\mbox{\boldmath $\varphi$}}_{\boldsymbol{\theta}}&=\tilde{\mbox{\boldmath $\theta$}}, \\
\dot{\mbox{\boldmath $\varphi$}}_{\mathbf{x}}&=\tilde{\mathbf{x}}.
\end{aligned} \right. 
\end{equation}
is linear. As with the system $\mathbf{S_1}$, $\mathbf{S_2}$ should be integrated several times on some interval $[0,\beta h]$, where $\beta$ is $b_1,\, b_2,\dots,b_{m}$, i.e. one of the coefficients of the splitting method. Using \eqref{controls1}, \eqref{controls2}, \eqref{controls3} and the variation of constants formula we obtain the following explicit expressions for the solution of \eqref{eq:vflows2}
\begin{align*}
Q(\gamma) &= Q_0,\\
\mbox{\boldmath $\varphi$}_{\boldsymbol{\theta}}(\gamma)&=\mbox{\boldmath $\varphi$}_{\boldsymbol{\theta},0}+\gamma\, \tilde{\mbox{\boldmath $\theta$}}_0, \hskip0.5cm \tilde{\mbox{\boldmath $\theta$}}_0:=\mbox{\boldmath $\theta$}_0-\mbox{\boldmath $\theta$}_{ref}, \\[0.2cm]
\left[\begin{array}{c}
\mathbf{m}(\gamma)\\
\mbox{\boldmath $\mu$}(\gamma)
\end{array}\right]&=e^{\gamma A}\, \left[\begin{array}{c}
\mathbf{m}_0\\
\mbox{\boldmath $\mu$}_0
\end{array}\right] +\gamma \,\phi_1(\gamma A)\left[\begin{array}{c}
\mathbf{w}_{r,1}\\
\mathbf{0}
\end{array}\right]+\gamma^2\,\phi_2(\gamma A)\left[\begin{array}{c}
\mathbf{w}_{r,2}\\
\mathbf{0}
\end{array}\right],\\
\left[\begin{array}{c}
\mathbf{p}(\gamma)\\
\mathbf{x}(\gamma)\\
\mbox{\boldmath $\varphi$}_{\mathbf{x}}(\gamma)
\end{array}\right]&=e^{\gamma B}\, \left[\begin{array}{c}
\mathbf{p}_0\\
\mathbf{x}_0\\
\mbox{\boldmath $\varphi$}_{\mathbf{x},0}
\end{array}\right] +\gamma \,\phi_1(\gamma B)\left[\begin{array}{c}
\mathbf{w}_t\\
\mathbf{0}\\
-\mathbf{x}_{ref}
\end{array}\right],\\
\end{align*}
where
{\scriptsize
$$A:=\left[\begin{array}{cc}
-(D_rT^{-1}+Q_0^T\Pi_e^{-T}K_d^r\Pi_e^{-1}Q_0T^{-1}) & \widehat{q_3}\\[0.3cm]
\mathcal{G}\widehat{q_3}T^{-1} & O\\
\end{array}\right]
,
$$
$$
B:=
\left[\begin{array}{ccc}
-(D_tm_v^{-1}+Q_0^TK_d^tQ_0m_v^{-1}), & -(Q_0^T\mathcal{A}+Q_0^TK_p^t) & -Q_0^TK_i^t \\[0.3cm]
Q_0m_v^{-1} & O & O\\[0.2cm]
O & I  & O
\end{array}\right],$$
}
\begin{align*}
\mathbf{w}_{r,1}& :=-Q_0^T\Pi_e^{-T}(K_p^r \tilde{\mbox{\boldmath $\theta$}}_0+K_i^r\mbox{\boldmath $\varphi$}_{\boldsymbol{\theta},0}),\\
\mathbf{w}_{r,2}& :=-Q_0^T\Pi_e^{-T}K_i^r\tilde{\mbox{\boldmath $\theta$}}_0, \\
\mathbf{w}_t& := Q_0^T(K_p^t\,\mathbf{x}_{ref}),
\end{align*}
and 
$$\phi_k(z) = \frac{1}{(k - 1)!} \int_0^1 e^{z (1-x)} x^{k-1} \,dx, \quad k=1,2.$$ 
From the variation of constants formula we deduce
{\small
$$\gamma \,\phi_1(\gamma A)\left[\begin{array}{c}
\mathbf{w}_{r,1}\\
\mathbf{0}
\end{array}\right]+\gamma^2\,\phi_2(\gamma A)\left[\begin{array}{c}
\mathbf{w}_{r,2}\\
\mathbf{0}
\end{array}\right]=-\int_0^{\gamma}e^{(\gamma-\sigma)A}Q_0^T\Pi_e^{-T}(K_p^{r}\tilde{\mbox{\boldmath $\theta$}}
 + K_i^{r} \mbox{\boldmath $\varphi$}_{\boldsymbol{\theta}}(\sigma))\, d\sigma,$$
 }
 where $\tilde{\mbox{\boldmath $\theta$}}$ does not depend on $\sigma$.

%% file: Chapters/2_Vessel_model.tex
\def\Z3x3{0{\ensuremath{_{3\times3}}}}
\def\I3x3{I{\ensuremath{_{3\times3}}}}

\section{A vessel rigid body model}
\label{model}

We give now more details about the system \eqref{eq:vesseleq} presented in the introduction. We will eventually show that it can be cast in the format of the system \eqref{eq:xiDiff}.
The generalized velocity vector $\mbox{\boldmath $\nu$}=
[\mathbf{v}^T, \mbox{\boldmath $\omega$}^T]^T \in \mathbf{R}^6$ lies in the body frame, where $\mathbf{v},\mbox{\boldmath $\omega$ } \in \mathbf{R}^3$ denote respectively the linear velocity and the angular velocity. The generalized position vector $\mbox{\boldmath $\eta$} = [\mathbf{x}^T, 
\mbox{\boldmath $\theta$}^T]^T \in \mathbf{R}^6$ lies in the spatial frame. Here $\mathbf{x}\in\mathbf{R}^3$ is the position vector and the components of $\mbox{\boldmath $\theta$}=[\phi,\theta,\psi]^T$ are the Euler angles, which provide a local representation of the attitude of the vessel. The attitude of the body evolves on the Lie group $SO(3)$, which is a manifold. The elements of $SO(3)$ can be represented as $3\times 3$ rotation matrices. If $Q$ is sufficiently close to the identity, then  the conversion between Euler angles and rotation matrices is 
\begin{equation}
\label{EulerangtoRot}
Q=\Psi(\mbox{\boldmath $\theta$})=R_{z,\psi }\,R_{y,\theta}\,R_{x,\phi},\quad \mbox{\boldmath $\theta$}=[\phi,\theta,\psi]^T,
\end{equation}
with
{\small
$$R_{x,\phi}=\left[
\begin{array}{ccc}
1 & 0 & 0 \\
0 & \cos(\phi) &-\sin(\phi)\\
0 & \sin(\phi) & \cos(\phi)
\end{array}
\right],\, R_{y,\theta}=\left[ \begin{array}{ccc}
\cos(\theta) & 0 & \sin(\theta) \\
0 & 1 &0\\
 -\sin(\theta) & 0 & \cos(\theta)
\end{array}
\right],\, 
R_{z,\psi}=\left[ \begin{array}{ccc}
\cos(\psi) & -\sin(\psi) & 0 \\
 \sin(\psi) & \cos(\psi) & 0\\
0 & 0 &1
\end{array}
\right], $$
}and $\Psi$ is a local diffeomorphism, i.e. $\Psi$ is invertible for $Q$ close to the identity.
Notice that $Q$ can always be transported close to the identity by multiplying with a rotation $Q_0^{-1}\approx Q^{-1}$ so that $Q_0^{-1}Q=\Psi(\mbox{\boldmath $\theta$})$.
The kinematic equation associated with \eqref{eq:vesseleq} is
\begin{equation}
\label{eq:kinematics}
\dot{\mbox{\boldmath $\eta$}} = J(\mbox{\boldmath $\eta$})
\mbox{\boldmath $\nu$},\qquad J(\mbox{\boldmath $\eta$}) = \left[\begin{array}{cc} Q & \Z3x3 \\
\Z3x3 & \Pi_e^{-1}Q \end{array} \right],
\end{equation}

\noindent with

$$
\Pi_e = \left[\begin{array}{ccc} \cos(\theta)\cos(\psi) & -\sin(\psi) & 0 \\
\cos(\theta)\sin(\psi) & \cos(\psi) & 0 \\
-\sin(\theta) & 0 & 1 \end{array} \right],
$$
see also \cite{marinecontrolsystems}, \cite{safstrom09mas}. 
We then get the kinematic equation $\dot{\mbox{\boldmath $\theta$}}=\Pi_e^{-1}Q\,\mbox{\boldmath $\omega$}$ for the Euler angles, which can be obtained by differentiating \eqref{EulerangtoRot} and is also valid locally. The matrix $\Pi_e$ is not invertible for $\theta=\pm \frac{\pi}{2}$ and $\psi=\pm \frac{\pi}{2}$. To avoid these singularities, we rewrite this kinematic equation 
as an equation for $Q$ 
\begin{equation}
\label{eq:attitudeSO3}
\dot{Q}=Q\,\widehat{\mbox{\boldmath $\omega$}}.
\end{equation}
This equation is globally defined on $SO(3)$.
The rotation matrix $Q$ transforms vectors in the body fixed frame $\{b\}$  
to vectors in the spatial frame $\{s\}$ \footnote{$Q:=R_b^e$ in the notation of \cite{perez07kmf} and \cite{marinecontrolsystems}.}.
Here
$$\widehat{\mbox{\boldmath $\,\,$}}:\mathbb{R}^3 \rightarrow \mathfrak{so}(3),\qquad \widehat{\mbox{\boldmath $\omega$}}=\left[
\begin{array}{ccc}
0 & -\omega_3 & \omega_2 \\
\omega_3 & 0 &-\omega_1\\
-\omega_2 & \omega_1 & 0
\end{array}
\right],$$
is the hat-map, and  $\mathfrak{so}(3)$ is the Lie algebra of $SO(3)$, consisting of $3\times 3$ skew-symmetric matrices, for further details see \cite{muller}.

Following \cite{lamb} (p. 151) and \cite{Kirchhoff}, equations \eqref{eq:vesseleq} can be obtained defining  the kinetic energy of the system to be
$$K=\frac{1}{2}\mbox{\boldmath $\nu$}^TM\mbox{\boldmath $\nu$},$$ leading to Kirchhoff's equations
\begin{eqnarray}
\label{kirchoff1}
\frac{d}{dt}\frac{\partial K}{\partial \mathbf{v}}& = & \frac{\partial K}{\partial \mathbf{v}} \times \mbox{\boldmath $\omega$}+ F_v,\\
\label{kirchoff2}
\frac{d}{dt}\frac{\partial K}{\partial \mbox{\boldmath $\omega$}}& = & \frac{\partial K}{\partial \mbox{\boldmath $\omega$}} \times \mbox{\boldmath $\omega$}+\frac{\partial K}{\partial\mathbf{v}}\times \mathbf{v}+ F_{\omega},
\end{eqnarray}
where $[\mathbf{p}^T, \mathbf{m}^T]^T:=\frac{\partial K}{\partial \mbox{\boldmath $\nu$}}=M\mbox{\boldmath $\nu$}$, and with $F_v$ and $F_{\omega}$ the external force and torque respectively. These include the damping forces, the gravitational and buoyancy forces, the control inputs, the forces due to the pipe, and environmental forces.%
In absence of external forces, the obtained equations have a Lie-Poisson structure, and can be derived by variational methods \cite{holm} (p.129). We refer to \cite{marsden99IMS} (p. 421) for the general description of the Euler-Poincare reduction on Lie groups (the relevant Lie group in this case is $SE(3)$). For formulations obtained applying the Hamilton-Pontriagin principle, see  \cite{BM}, and for the inclusion of external forces see \cite{marsdenwest} (p. 421).

If we set
\begin{equation}
\label{RBmomenta}
\mathbf{p}:=\frac{\partial K}{\partial \mathbf{v}},\qquad   \mathbf{m}:= \frac{\partial K}{\partial \mbox{\boldmath $\omega$}}, \quad 
\left[  \begin{array}{c}
\mathbf{p}\\
 \mathbf{m}
\end{array}
\right]= M \, \mbox{\boldmath $\nu$},
\end{equation}
the general form of the skew-symmetric Coriolis-centripetal matrix is
\begin{equation}
\label{Coriolisgeneral}
C(\mbox{\boldmath $\nu$}) = -\left[
\begin{array}{cc}
\Z3x3 & \widehat{\frac{\partial K}{\partial \mathbf{v}}}\\[0.1cm]
\widehat{\frac{\partial K}{\partial \mathbf{v}}} & \widehat{\frac{\partial K}{\partial \mbox{\boldmath $\omega$}}}
\end{array}
\right]=-\left[
\begin{array}{cc} 
\Z3x3 & \hat{\mathbf{p}} \\ 
 \hat{\mathbf{p}} & \hat{\mathbf{m}} \end{array} \right].
\end{equation}

Since $K$ is a homogeneous quadratic polynomial in  the components of $\mbox{\boldmath $\nu$}$, only the symmetric part of $M$ plays a role, and we can assume $M$ to be symmetric. The matrix $M$ can be split into a rigid body part $M_{RB}$ and a part for
added mass $M_A$. The same can be done for the kinetic energy $K=K_{RB}+K_{A}$ and the Coriolis-centripetal matrix $C(\mbox{\boldmath $\nu$})=C_{RB}(\mbox{\boldmath $\nu$})+C_A(\mbox{\boldmath $\nu$})$. 
Letting the origin of the body fixed frame coincide with the center of gravity in our model, the rigid body system inertia matrix $M_{RB}$ has the simple form
\begin{equation}
\label{MRB}
M_{RB} = \left[\begin{array}{cc} m_v\I3x3 & \Z3x3 \\
\Z3x3 & T \end{array} \right],
\end{equation}
where $m_v$ is the mass of the vessel, $\I3x3$ is the 3-dimensional identity
matrix, and $T$ is the inertia tensor.
The added mass term $M_A$ is a symmetric matrix, and 
we notice that when $M_A$ is set to zero, due to the structure of $M_{RB}$, the term $\frac{\partial K}{\partial\mathbf{v}}\times \mathbf{v}$ in \eqref{kirchoff2} vanishes. 

With this simplification, we get 
$$
\frac{\partial K}{\partial \mathbf{v}}=m_v  \mathbf{v},\qquad    \frac{\partial K}{\partial \mbox{\boldmath $\omega$}}=T \mbox{\boldmath $\omega$},
$$
and from \eqref{kirchoff1} and \eqref{kirchoff2} we deduce that  the rigid body Coriolis-centripetal matrix in this case can be rewritten in the form

$$
C_{RB}(\mbox{\boldmath $\nu$}) = \left[\begin{array}{cc} m_v\hat{\mbox{\boldmath $\omega$}}
& \Z3x3 \\ \Z3x3 & -\widehat{T\mbox{\boldmath $\omega$}} \end{array} \right].
$$
In the general case, 
$M$ is symmetric and 
$C(\mbox{\boldmath $\nu$})$ is  given by \eqref{Coriolisgeneral}, with $\mathbf{p}$ and $\mathbf{m}$  given by \eqref{RBmomenta}. We will here make the additional assumption that $M$ is block diagonal, (see \cite{marinecontrolsystems} p. 98 about this assumption), and that it is diagonal when the origin of the body frame is in the center of gravity. 


We have so far accounted for the first two terms in the equations  \eqref{eq:vesseleq}, and we will in the sequel describe the remaining forces and moments.

The term $\mathbf{g}(\mbox{\boldmath $\eta$})$ from Equation (\ref{eq:vesseleq}) takes restoring
forces and moments into account. In the model adopted in \cite{jensen09mac},
$\mathbf{g}(\mbox{\boldmath $\eta$})$, which lies in the body frame, is given as

$$
\mathbf{g}(\mbox{\boldmath $\eta$}) = \left[\begin{array}{c} Q^T \mathbf{g}_t^s \\ Q^T \mathbf{g}_r^s \end{array}
\right],
$$

\noindent where the superscript $s$ denotes that the vector lies in the spatial frame. 

Again following \cite{jensen09mac}, we have for the rotational part

$$
\mathbf{g}_r^s = (Q\mathbf{r}_r^b)\times(m_vg\mathbf{e}_3),
$$

\noindent where $g$ is the gravitational acceleration and $\mathbf{r}_r^b$ is the moment arm
in the body frame. The latter can be expressed as 

$$
\mathbf{r}_r^b = \left[\begin{array}{c} \overline{GM}_L (Q\mathbf{e}_1)^T\mathbf{e}_3 \\
\overline{GM}_T (Q\mathbf{e}_2)^T\mathbf{e}_3 \\ 0 \end{array}\right],
$$

\noindent where $\overline{GM}_L$ and $\overline{GM}_T$ are the longitudinal and the
transverse metacentric heights of the vessel, respectively.

The vector $\mathbf{g}_t^s$ accounts for buoyancy forces, and can be modelled as
$$
\mathbf{g}_t^s = g \rho_w A_{wp}(z-z_{eq}) \mathbf{e}_3,
$$
following an approach similar to that in \cite{jensen10anp}. Here $\rho_w$ is the density of water, $A_{wp}$ is the water plane area and $z_{eq}$ is the equilibrium water level. These forces can be directly included by adding a potential energy term to the kinetic energy $K=\frac{1}{2}\mbox{\boldmath $\nu$}^TM\mbox{\boldmath $\nu$}$, when considering the Lie group models of \cite{holm}, \cite{marsden99IMS} and \cite{BM}.

For the positive definite damping matrix $D(\mbox{\boldmath $\nu$})$, we also make use of a common simplifying assumption that the matrix is diagonal when the origin of the fixed body coordinate system coincides with the center of gravity. 
 We therefore represent this matrix as
$$
D(\mbox{\boldmath $\nu$}) = \left[\begin{array}{cc} D_t(\mathbf{v})
& \Z3x3 \\ \Z3x3 & D_r (\mbox{\boldmath $\omega$}) \end{array} \right],
$$
where $D_t,D_r \in  \mathbb{R}^{3\times 3}$ are positive definite diagonal matrices.
\subsection{Right hand side forces}\label{controls}

Let us now consider  the term $\mbox{\boldmath $\tau$} $ in \eqref{eq:vesseleq}.
Following the model presented in \cite{jensen10anp} we set

\begin{equation}
\begin{array}{rcl}
\mbox{\boldmath $\tau$} &=& -J^T(\mbox{\boldmath $\eta$}) \mbox{\boldmath $\tau$}_{PID}, \\
\mbox{\boldmath $\tau$}_{PID} &=& K_p \tilde{\mbox{\boldmath $\eta$}} + K_d \dot{\tilde{\mbox{\boldmath $\eta$}}}
+ K_i\int_{t_0}^{t}\! \tilde{\mbox{\boldmath $\eta$}}(\sigma) \,d\sigma,
\label{eq:control_tau}
\end{array}
\end{equation}

\noindent where the matrices

\begin{equation}
K_p=\left[\begin{array}{cc}K_p^{t} & \Z3x3 \\ \Z3x3 & K_p^{r}\end{array}\right], \hspace{4mm}
K_d=\left[\begin{array}{cc}K_d^{t} & \Z3x3 \\ \Z3x3 & K_d^{r}\end{array}\right],  \hspace{2mm} \text{and} \hspace{2mm}
K_i=\left[\begin{array}{cc}K_i^{t} & \Z3x3 \\ \Z3x3 & K_i^{r}\end{array}\right],
\label{eq:control_gains}
\end{equation}

\noindent are called controller gains, and
$\tilde{\mbox{\boldmath $\eta$}} := \mbox{\boldmath $\eta$} - \mbox{\boldmath $\eta$}_{ref}$.
%
%
%
%

This formulation of the control forces is based on the use of Euler angles, $\mbox{\boldmath $\theta$}$, to locally parametrize the rotation matrix $Q$.
Away from singularities, i.e. for $Q$ close to the identity, we have $\mbox{\boldmath $\theta$}=\Psi^{-1}( Q)$, with $\Psi$ given by \eqref{EulerangtoRot}.  A reformulation of the control forces using only $Q$ and the angular velocity $\omega$ will require a different design of the controller gains \eqref{eq:control_gains} and will not be pursued here.

In the case where $K_i$ is different from the zero matrix, we introduce a new set of unknowns
$$
\mbox{\boldmath $\varphi$}:=\int_{t_0}^t \tilde{\boldsymbol{\eta}} (\sigma)\, d\sigma,
$$
and add a corresponding new set of equations to the system \eqref{eq:vesseleq}
$$\dot{\mbox{\boldmath $\varphi$}}= \tilde{\boldsymbol{\eta}}(t),\qquad {\mbox{\boldmath $\varphi$}}(t_0)=0.$$
The integral term in $\mbox{\boldmath $\tau$}_{PID}$ can then be written as
$$\int_{t_0}^{t}\! \tilde{\boldsymbol{\eta}}(\sigma) \,d\sigma={\mbox{\boldmath $\varphi$}}, \qquad {\mbox{\boldmath $\varphi$}}=({\mbox{\boldmath $\varphi$}}_{\mathbf{x}}, {\mbox{\boldmath $\varphi$}}_{\boldsymbol{\theta}}),$$
(as with ${\boldsymbol{\theta}}$, the equation for ${\mbox{\boldmath $\varphi$}}_{\boldsymbol{\theta}}$ is valid locally). %
Note that the system of equations is then autonomous, i.e. does not explicitly depend on $t$.
Though not required, we let $\mathbf{w} = \mathbf{0}$ for simplicity in what follows. 
In \cite{jensen10anp} the pipe is assumed to be fixed to the center of gravity of the vessel. The forces and moments from the pipe on the vessel, $\mathbf{\mathcal{X}}$, depend on the stress resultant and stress couple of the pipe.  Since the modeling and simulation of the pipe structure is not the focus of this paper, these forces are not taken into account. 
 %
\subsection{The set of equations}

Using the explicit expressions for the right hand side forces, the system \eqref{eq:vesseleq} can be written as 

\begin{align}
\dot{\mathbf{p}} &= \mathbf{p}\times \mbox{\boldmath $\omega$} -\left(D_t(\mathbf{v})\mathbf{v}+Q^T\mathbf{g}_t^s(\mathbf{x})-\mbox{\boldmath $\tau$}_t\right), \nonumber  \\ 
\dot{\mathbf{m}} &= \mathbf{m} \times \mbox{\boldmath $\omega$}+\mathbf{p}\times \mathbf{v} - \left(D_r(\mbox{\boldmath $\omega$})\mbox{\boldmath $\omega$}+Q^T\mathbf{g}_r^s(Q)-\mbox{\boldmath $\tau$}_r\right), \label{eq:vessel}\\
\dot{Q}&=Q\,\widehat{\mbox{\boldmath $\omega$}}, \nonumber \\
\dot{\mathbf{x}} &= Q\,\mathbf{v}, \nonumber
\end{align}
with
\begin{align}
\mbox{\boldmath $\tau$}_r  &= -\left(\Pi_e^{-1}Q\right)^T(K_p^{r}\tilde{\mbox{\boldmath $\theta$}}
+ K_d^{r}\dot{\tilde{\mbox{\boldmath $\theta$}}} + K_i^{r} \mbox{\boldmath $\varphi$}_{\boldsymbol{\theta}} ), \label{controls1} \\
\boldsymbol{\tau}_t &= -Q^T \left(K_p^{t}\tilde{\mathbf{x}}
+ K_d^{t}\dot{\tilde{\mathbf{x}}} + K_i^{t} \mbox{\boldmath $\varphi$}_{\mathbf{x}}\right), \label{controls2} \\
\tilde{\mbox{\boldmath $\theta$}} &= \Psi^{-1}(Q)- \mbox{\boldmath $\theta$}_{ref}, 
\hspace{6mm}
\dot{\tilde{\mbox{\boldmath $\theta$}}} = \dot{\mbox{\boldmath $\theta$}} = \Pi_e^{-1}Q\, \mbox{\boldmath $\omega$} ,
\hspace{6mm}
\tilde{\mathbf{x}} = \mathbf{x}-\mathbf{x}_{ref}, \label{controls3}
\end{align}
and
\begin{align}
\dot{\mbox{\boldmath $\varphi$}}_{\boldsymbol{\theta}}&=\tilde{\mbox{\boldmath $\theta$}}, \nonumber\\
\dot{\mbox{\boldmath $\varphi$}}_{\mathbf{x}} &=\tilde{\mathbf{x}}. \nonumber
\end{align}

An alternative version of these equations, which uses Euler parameters to represent the attitude, is given in Appendix \ref{quatEq}.

\subsection{Passivity and splitting of the vessel equations}
For the vessel model, we now define $\mbox{\boldmath $\xi$} := [
\mathbf{p}^T, \mathbf{m}^T, \mathbf{q}_1^T, \mathbf{q}_2^T, \mathbf{q}_3^T, \mbox{\boldmath $\mu$}^T, \mathbf{x}^T ]$
with
$$
\mathbf{p} := m_v \mathbf{v}, \quad\mathbf{m} := T\boldsymbol{\omega},  \quad \mathbf{q}_i := Q^T\mathbf{e_i},\quad \mbox{\boldmath $\mu$}:=\mathcal{G}\mathbf{q}_3,
$$
with $m_v$ and $T$ diagonal and invertible matrices.
We take $H = H(\mbox{\boldmath $\xi$})$ to be the Hamiltonian of the system, given as the sum of the kinetic energy $K$ and the potential energy $U$. 
Defining
{\small
$$\mathcal{A}  := \mathrm{diag}(0,0,c), \quad \mathcal{G}  := m_v g\, \mathrm{diag}(\overline{GM}_L,\overline{GM}_T,0),\quad \tilde{\mathcal{G} } := \frac{1}{m_v g}\, \mathrm{diag}((\overline{GM}_L)^{-1},(\overline{GM}_T)^{-1},0),$$
}
and $c := g \rho_w A_{wp},$
we have
\begin{equation}
\label{Hamiltonian}
\begin{array}{lcl}K =\frac{1}{2}\mathbf{m}^T\ T^{-1} \mathbf{m} + \frac{1}{2}\mathbf{p}^T m_v^{-1}\mathbf{p},\\
U= \frac{1}{2}\mbox{\boldmath $\mu$}^T\tilde{\mathcal{G}}\mbox{\boldmath $\mu$}+\frac{1}{2}\mathbf{q}_1^T\mathbf{q}_1+\frac{1}{2}\mathbf{q}_2^T\mathbf{q}_2 + \frac{1}{2}\mathbf{q}_3^T\mathbf{q}_3+ \frac{1}{2}\mathbf{x}^T\mathcal{A}\mathbf{x},\\
H=K+U.
\end{array}
\end{equation}
Notice that the following identities hold
\begin{equation}
\label{gtsgrs}
\mathbf{g}_t^s=\mathcal{A}\, \mathbf{x},\quad Q^T\mathbf{g}_r^s=\mbox{\boldmath $\mu$}\times \mathbf{q}_3.
\end{equation}
We will in the next proposition replace the matrix equation $\dot{Q}=Q\hat{\boldsymbol{\omega}}$ in \eqref{eq:vessel} with three vector equations, one for each column $\mathbf{q}_i$, $i=1,2,3$, of $Q^T$.

\begin{Proposition}
\label{portHamSys}
The system \eqref{eq:vessel} can be written in the form
\begin{equation}
\label{eq:xiDiff1}
\dot{\mbox{\boldmath $\xi$}} = S(\mbox{\boldmath $\xi$})\nabla H(\mbox{\boldmath $\xi$}) -\tilde{I}\left(D(\mbox{\boldmath $\xi$})\mbox{\boldmath $\nu$} - \tilde{\mbox{\boldmath $\tau$}}\right),
\end{equation}
with
{\small
\begin{equation}
\label{eq:skewMat}
S(\mbox{\boldmath $\xi$}) := \begin{bmatrix} 
0_{3 \times 3} & \widehat{\mathbf{p}} & 0_{3 \times 3} & 0_{3 \times 3} & 0_{3 \times 3} &  0_{3 \times 3} &-Q^T \\ 
 \widehat{\mathbf{p}} &\widehat{\mathbf{m}} & 0_{3 \times 3} &0_{3 \times 3} & -\widehat{\mbox{\boldmath $\mu$}} & 0_{3 \times 3}  & 0_{3 \times 3}\\ 
 0_{3 \times 3} &  0_{3 \times 3} & -\widehat{T^{-1}\mathbf{m}} & 0_{3 \times 3} & 0_{3 \times 3} &0_{3 \times 3} & 0_{3 \times 3}\\
 0_{3 \times 3} & 0_{3 \times 3} & 0_{3 \times 3} &  -\widehat{T^{-1}\mathbf{m}} & 0_{3 \times 3}  &   0_{3 \times 3} & 0_{3 \times 3}\\
 0_{3 \times 3} & -\widehat{\mbox{\boldmath $\mu$}} &  0_{3 \times 3} &   0_{3 \times 3} &   -\widehat{T^{-1}\mathbf{m}} &   -\widehat{T^{-1}\mathbf{m}}\,\mathcal{G}  & 0_{3 \times 3}\\
 0_{3 \times 3} & 0_{3 \times 3} & 0_{3 \times 3} & 0_{3 \times 3} & -\mathcal{G}\,\widehat{T^{-1}\mathbf{m}} & 0_{3 \times 3} & 0_{3 \times 3}\\
Q & 0_{3 \times 3} &  0_{3 \times 3} &   0_{3 \times 3} &   0_{3 \times 3} &   0_{3 \times 3} & 0_{3 \times 3}
\end{bmatrix}, \ \ 
\tilde{I} := \begin{bmatrix} 
I{\ensuremath{_{6\times6}}} \\
0{\ensuremath{_{15\times6}}}
\end{bmatrix},
\end{equation}
}
$$D(\mbox{\boldmath $\xi$}):=D+\left[\begin{array}{cc} Q  & \Z3x3 \\
\Z3x3 & \Pi_e^{-1}Q \end{array} \right]^TK_d \left[\begin{array}{cc} Q & \Z3x3 \\
\Z3x3 & \Pi_e^{-1}Q\end{array} \right],$$
and

\begin{eqnarray*}
\tilde{\mbox{\boldmath $\tau$}}&=&\mbox{\boldmath $\tau$}-\left[\begin{array}{cc} Q & \Z3x3 \\
\Z3x3 & \Pi_e^{-1}Q \end{array} \right]^TK_d \left[\begin{array}{cc} Q & \Z3x3 \\
\Z3x3 & \Pi_e^{-1}Q \end{array} \right] \mbox{\boldmath $\nu$} \\[0.2cm]
&=&\left[\begin{array}{cc} Q & \Z3x3 \\
\Z3x3 & \Pi_e^{-1}Q \end{array} \right]^T \left[\begin{array}{c}
K_p^{r}\tilde{\mbox{\boldmath $\theta$}}+K_i^{r} \mbox{\boldmath $\varphi$}_{\boldsymbol{\theta}}\\
K_p^{t}\tilde{\mathbf{x}}+K_i^{t} \mbox{\boldmath $\varphi$}_{\mathbf{x}}
\end{array}\right]
\end{eqnarray*}

with 

\begin{align*}
\tilde{\mbox{\boldmath $\theta$}} &= \mbox{\boldmath $\theta$}- \mbox{\boldmath $\theta$}_{ref}, 
\hspace{6mm}
\dot{\tilde{\mbox{\boldmath $\theta$}}} = \dot{\mbox{\boldmath $\theta$}} = \Pi_e^{-1}Q\, \mbox{\boldmath $\omega$} ,
\hspace{6mm}
\tilde{\mathbf{x}} = \mathbf{x}-\mathbf{x}_{ref},
\hspace{6mm}
\dot{\mbox{\boldmath $\varphi$}}_{\boldsymbol{\theta}} =\tilde{\mbox{\boldmath $\theta$}},
\hspace{6mm}
\dot{\mbox{\boldmath $\varphi$}}_{\mathbf{x}} =\tilde{\mathbf{x}}.
\end{align*}

\end{Proposition}
\begin{proof}
The proof follows by a direct calculation of the gradient of $H$. 
\end{proof}

The obtained system is completely analog to  \eqref{eq:xiDiff}, and its passivity follows from Proposition~\ref{passivity}.
To obtain a passivity preserving splitting of the vessel model, we can proceed  by splitting $S$ in \eqref{eq:skewMat} as the sum of the two following skew symmetric terms %
{\scriptsize
$$
S_1(\mbox{\boldmath $\xi$}) := \begin{bmatrix} 
0_{3 \times 3} & \widehat{\mathbf{p}} & 0_{3 \times 3} & 0_{3 \times 3} & 0_{3 \times 3} &  0_{3 \times 3} &0_{3 \times 3} \\ 
\widehat{\mathbf{p}}  &\widehat{\mathbf{m}} & 0_{3 \times 3} &0_{3 \times 3} & 0_{3 \times 3} & 0_{3 \times 3}  & 0_{3 \times 3}\\ 
 0_{3 \times 3} &  0_{3 \times 3} & -\widehat{T^{-1}\mathbf{m}} & 0_{3 \times 3} & 0_{3 \times 3} &0_{3 \times 3} & 0_{3 \times 3}\\
 0_{3 \times 3} & 0_{3 \times 3} & 0_{3 \times 3} &  -\widehat{T^{-1}\mathbf{m}} & 0_{3 \times 3}  &   0_{3 \times 3} & 0_{3 \times 3}\\
 0_{3 \times 3} & 0_{3 \times 3} &  0_{3 \times 3} &   0_{3 \times 3} &   -\widehat{T^{-1}\mathbf{m}} &  0_{3 \times 3}  & 0_{3 \times 3}\\
 0_{3 \times 3} & 0_{3 \times 3} & 0_{3 \times 3} & 0_{3 \times 3} & 0_{3 \times 3} & 0_{3 \times 3} & 0_{3 \times 3}\\
0_{3 \times 3} & 0_{3 \times 3} &  0_{3 \times 3} &   0_{3 \times 3} &   0_{3 \times 3} &   0_{3 \times 3} & 0_{3 \times 3}
\end{bmatrix},
$$
$$
S_2(\mbox{\boldmath $\xi$}) := \begin{bmatrix} 
0_{3 \times 3} & 0_{3 \times 3} & 0_{3 \times 3} & 0_{3 \times 3} & 0_{3 \times 3} &  0_{3 \times 3} & -Q^T  \\ 
 0_{3 \times 3}&0_{3 \times 3} & 0_{3 \times 3} &0_{3 \times 3} & -\widehat{\mbox{\boldmath $\mu$}} & 0_{3 \times 3}  & 0_{3 \times 3}\\ 
 0_{3 \times 3} &  0_{3 \times 3} & 0_{3 \times 3} & 0_{3 \times 3} & 0_{3 \times 3} &0_{3 \times 3} & 0_{3 \times 3}\\
 0_{3 \times 3} & 0_{3 \times 3} & 0_{3 \times 3} &  0_{3 \times 3}& 0_{3 \times 3}  &   0_{3 \times 3} & 0_{3 \times 3}\\
 0_{3 \times 3} & -\widehat{\mbox{\boldmath $\mu$}} &  0_{3 \times 3} &   0_{3 \times 3} &  0_{3 \times 3} &   -\widehat{T^{-1}\mathbf{m}}\,\mathcal{G}  & 0_{3 \times 3}\\
 0_{3 \times 3} & 0_{3 \times 3} & 0_{3 \times 3} & 0_{3 \times 3} & -\mathcal{G}\,\widehat{T^{-1}\mathbf{m}} & 0_{3 \times 3} & 0_{3 \times 3}\\
Q  & 0_{3 \times 3} &  0_{3 \times 3} &   0_{3 \times 3} &   0_{3 \times 3} &   0_{3 \times 3} & 0_{3 \times 3}
\end{bmatrix} ,
$$
}
and then applying an energy-preserving/energy-dissipative integrator on each of the individual flows. The energy in this case is quadratic.  By Proposition~\ref{propositionpassivitysplitting}, %
such a splitting method preserves passivity. In the next section we consider the implementation of this method in a special case, when the splitting leads to two completely integrable flows.

\vspace{1em}


%% file: Chapters/5_Numerical_results.tex
\section{Numerical results}
\label{experiments}

In this section we report the results of some numerical experiments for %
\eqref{eq:vessel}. The parameter values used for the rigid body vessel model are given in Appendix \ref{ParamValues}. For the parametrization of the rotation matrices we have used unit quaternions, see  Appendix \ref{quatEq}. The reference solution has been computed with the classical Runge-Kutta method of order $4$ (RK4) with small step size $h = 10^{-6}$ in the order tests, Figure~\ref{order}, and $h = 10^{-4}$ in the other experiments. 

In Figure \ref{evo}, we show the features of the solution of the considered controlled system, and see that the desired equilibrium is attained. We plot the evolution of the generalized positional coordinates over the interval $t \in [0,200]$. The controls are turned on at $t=50$, \cite{jensen09mac}. Let us denote by $x$ and $y$ the first two components of the vector $\mathbf{x}$, and recall that $\psi$ is the third component of $\boldsymbol{\theta}$. The results show rapid convergence towards the reference values for $x$, $y$ and $\psi$. There is a slight initial overshoot due to the presence of the integral term, which will vanish over time. The overshoot can be reduced by 
further optimizing the choice of controller gains. In our experiments, choosing different values for the environmental forces, we observe that the integral term is in general necessary to ensure convergence to the set point, see also \cite{jensen09mac}. The uncontrolled variables show expected behaviour. The remaining angles have small initial oscillations which are damped towards 0. In the case of $\phi$ the damping is present, but too small to be observed from the plot. When the controls are turned on, the elevation $z$ moves away from the equilibrium, but it rapidly stabilises back to the desired value.

\begin{figure}
  \centering 
   \includegraphics[width=1\textwidth]{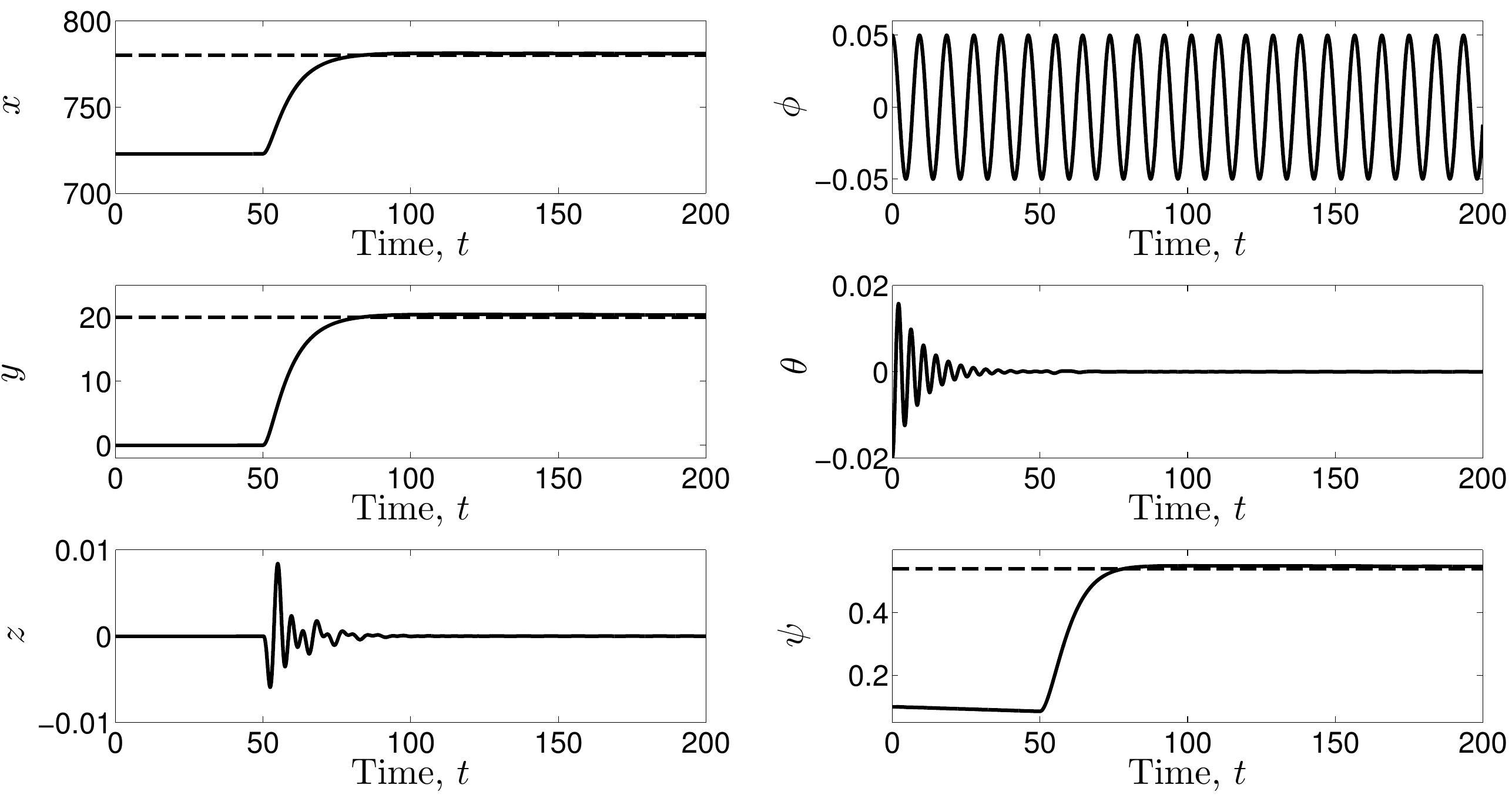}
\caption{Evolution of the generalized position variables over time with $t \in [0,200]$. The PID controller is turned on at $t=50$. All the controlled variables converge rapidly towards the reference values, with a small overshoot caused by the integral term in the controller.  \label{evo}}
  
\end{figure}

We now verify the order of accuracy of the numerical solutions obtained by the order 2, order 4 and order 6 splitting schemes, which we will refer to as SP2, SP4 and SP6, respectively. We define the relative error for the angular momentum at time $t_n$
$$
e^{[\mathbf{m}]} = \frac{\vectornorm{\mathbf{m}_n-\mathbf{m}(t_n)}_2}{\vectornorm{\mathbf{m}(t_n)}_2},
$$
where $\mathbf{m}(t_n)$
is the reference solution, and $\mathbf{m}_n \approx \mathbf{m}(t_n)$ is the approximation given by the numerical method whose accuracy we want to investigate. Similar relative errors $e^{[q]}$, $e^{[\mathbf{v}]} $ and $e^{[\mathbf{x}]} $ are considered for $q$, $\mathbf{v}$ and $\mathbf{x}$, respectively. Here $q$ is the unit quaternion representing the attitude $Q$. The errors are evaluated at the end of the time interval, $t\in [0,1]$. In Figure \ref{order} we plot the relative errors against the step size for step sizes $h_{k} =\frac{1}{2^k}$  with $k =0,1,\hdots,4$, and show that the correct order is attained by the methods.

%

\begin{figure}
  \includegraphics[width=0.85\textwidth]{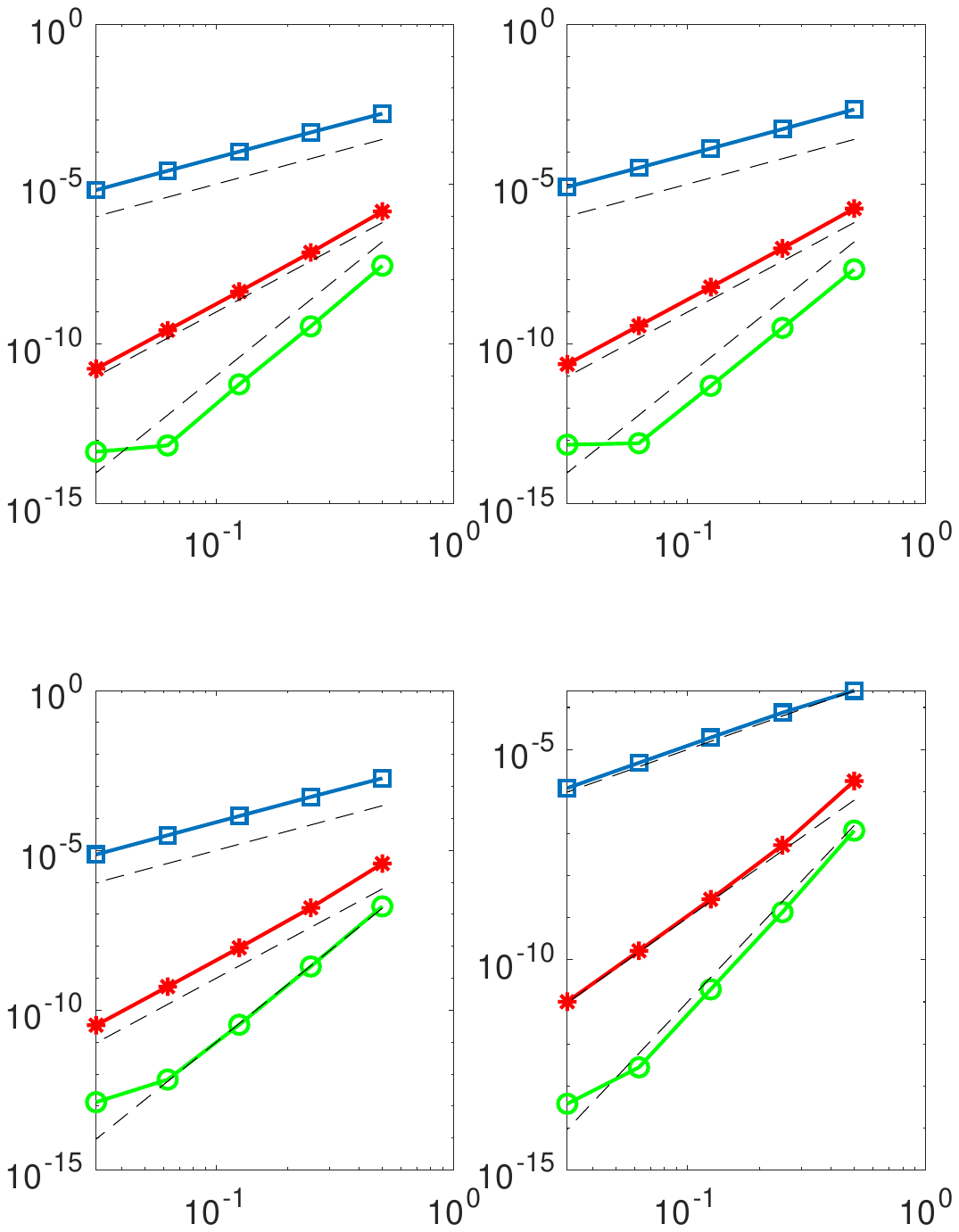}
\caption{Dependence of the relative endpoint error ($y$-axis) on the size of the integration step ($x$-axis) for \eqref{eq:vflows}. 
Results for SP2, SP4 
and SP6 
are plotted in logarithmic scale for $t \in [0,1]$, $h=\frac{1}{2^k}$ and $k=0,1,\hdots , 4$. Top: global error in angular momentum $\mathbf{m}$ (left) and attitude $q$ (right). Bottom: global error in linear momentum $\mathbf{p}$ and position $\mathbf{x}$. The expected order for the splitting schemes is obtained. \label{order}}
\end{figure}

Finally, in Figure \ref{EnErrorEvo}, we compare the norm of the scaled control input and the energy  function $H$, given by  \eqref{Hamiltonian} for the splitting methods SP2 and SP4, and two explicit Runge-Kutta methods: the Improved Euler method (a second order Runge-Kutta method\footnote{$y_{n+1}=y_n+\frac{h}{2}\big( f(y_n)+f(y_n+hf(y_n))\big)$}) and RK4. Methods of the same order are compared using relatively large step sizes, specifically $h=2/3$ (top figures) and $h=1.9375$ (bottom figures). The values from the splitting methods in this experiment cannot be distinguished from the corresponding values from the exact solution, while the values from the explicit Runge-Kutta methods differ noticeably. For larger values of the step-sizes the explicit Runge-Kutta methods are unstable. A more detailed analysis of the stability and of the error of splitting methods for similar problems in rigid body dynamics is under investigation. Preliminary results show that as $h$ increases, splitting methods suffer of order reduction, but they do not become unstable and the error remains bounded. This is similar to what shown in the case of Schr{\"o}dinger equations, see \cite{janke}. 

\begin{figure}
  \centering 
     \includegraphics[width=1\textwidth]{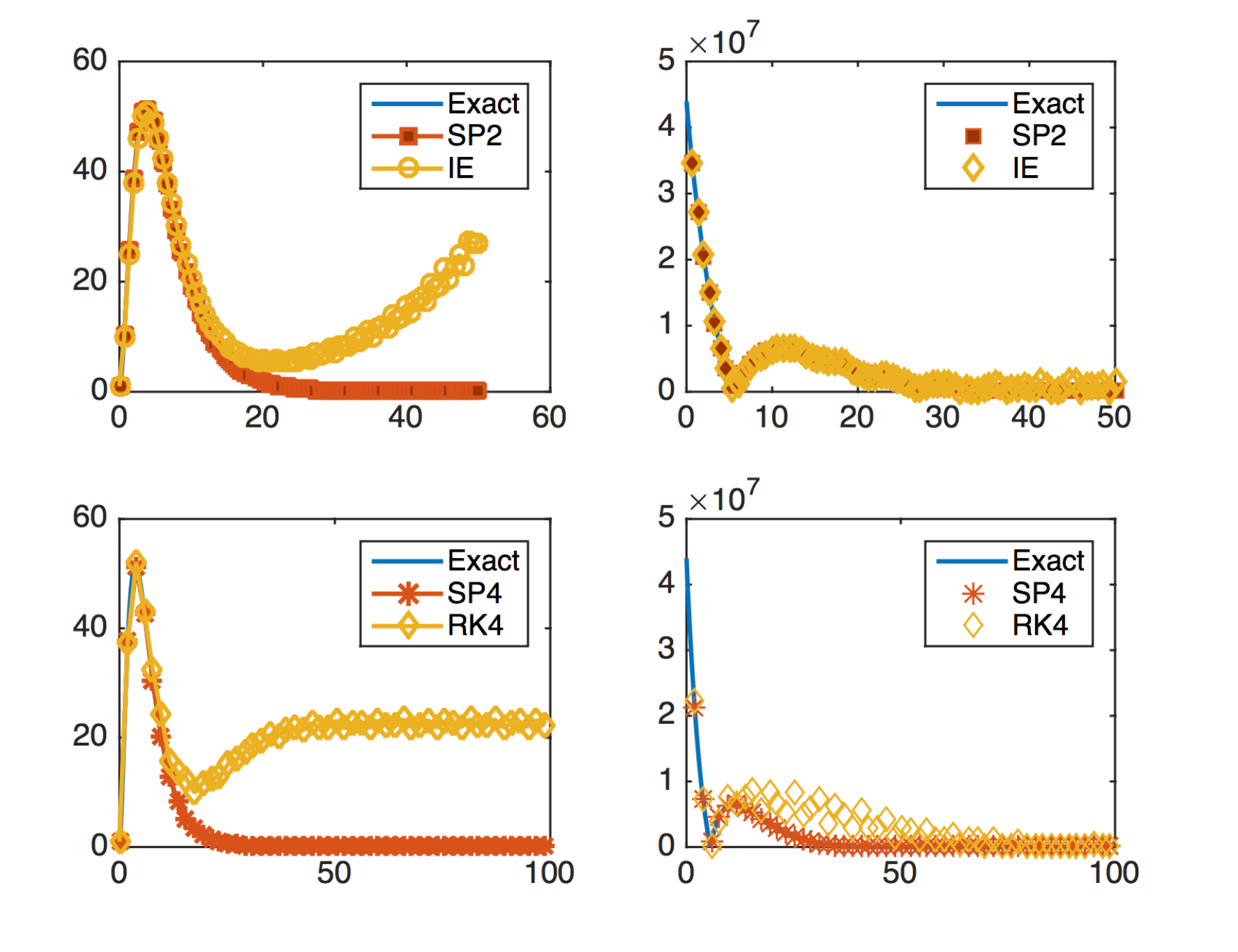}
\caption{Time versus the scaled energy $H(y_n)/H_0$ (left) and time versus the norm of the scaled rotational control $\left\|T^{-1}\boldsymbol{\tau}_r\right\|_2$ (right). The two top figures compare the methods of order 2: SP2 and Improved Euler, for time $t \in [0,50]$ and step size $h=2/3$. The two bottom figures compare methods of order $4$: SP4 and RK4 with step size $h = 1.9375$ on the time interval $[0,100]$. All methods are compared with the exact solution. The splitting methods follow the exact scaled energy and torque well. On the contrary, the values from the RK4 and of the Improved Euler method are seen to differ.  \label{EnErrorEvo}}
\end{figure}


%% file: Chapters/6_Concluding_remarks.tex
\section{Concluding remarks}
\label{conclusion}
We have proposed a class of splitting schemes for a system of controlled vessel rigid body equations in use in offshore marine operations. 
The discrete systems respect important qualitative features of the continuous system: we have proved that the numerical flow obtained by the splitting methods is input-output passive, and the numerical attitude  evolves on the Lie group $SO(3)$ as its continuous counterpart.
The benefit of the passivity-preserving schemes is that they satisfy a definition of passivity analogous to the one satisfied by the continuous equations. The control of the continuous system is designed so to guarantee the attainment of a certain desired equilibrium, but ultimately the systems are discretized with numerical integration techniques. The passivity preserving schemes ensure this same property at the discrete level by construction of the numerical methods.

With the proposed splitting schemes we achieve improved accuracy in the numerical solution. In our numerical experiments the methods show better energy behaviour compared to classical Runge-Kutta techniques.

The splitting methods considered have not been optimised with respect to global error, number of stages, i.e. elementary flows composed, and computational effort. 
This is future work. A reformulation of the input forces using only the attitude  and the angular velocity, and a general improvement of the controlled system will  also be considered in the future. 

%% file: Chapters/Appendix.tex
\appendix
\numberwithin{equation}{section}
\section{Euler parameters} \label{EulerPar}

We here review briefly the main properties of quaternions and Euler parameters. More information on this subject can be found in e.g. \cite{marsden99IMS}. The set of quaternions,
    \[
    \mathbb{H}:=\{q=(q_0,\mathbf{q})\in \mathbb{R}\times \mathbb{R}^3, \, \mathbf{q}=[q_1,q_2,q_3]^T 		\}\cong \mathbb{R}^4,
    \]
    is a strictly skew field \cite{baker02mg}.
    Addition and multiplication of two quaternions, $p=(p_0,\mathbf{p}),\,q=(q_0,\mathbf{q})\in \mathbb{H}$, are defined by
    \[
    p+q:=(p_0+q_0,\mathbf{p}+\mathbf{q}),
    \]
    and
    \begin{equation}
    \label{mult_q}
        pq:=(p_0q_0-\mathbf{p}^T \mathbf{q},p_0\mathbf{q}+q_0\mathbf{p}+\mathbf{p}\times \mathbf{q}).
    \end{equation}
For $q\neq (0,\mathbf{0})$ there exists an inverse
\[
    q^{-1}:=q^c/\| q\|^2,\quad
    \|q\| := \sqrt{q_0^2 + \|\mathbf{q}\|_2^2},
\]
where $q^c:=(q_0,-\mathbf{q})$ is the conjugate of $q$, such that
$qq^{-1}=q^{-1}q= e=(1,\mathbf{0})$.
In the sequel we will consider $q\in \mathbb{H}$ as a
vector $q=[q_0,q_1,q_2,q_3]^T\in \mathbb{R}^4$. The
multiplication rule (\ref{mult_q}) can then be expressed by means of
a matrix-vector product in $\mathbb{R}^4$. Namely,
    $pq=L(p)q=R(q)p$, where
    \begin{equation*}
    \label{LeftRightMult}
    L(p):=\left[
    \begin{array}{cc}
        p_0&-\mathbf{p}^T\\
        \mathbf{p} & p_0I+\widehat{\mathbf{p}}
    \end{array}\right], \quad
    R(q):=\left[
    \begin{array}{cc}
        q_0&-\mathbf{q}^T\\
        \mathbf{q} & q_0I-\widehat{\mathbf{q}}
    \end{array}\right],
    \end{equation*}
    and $I$ is the $3\times 3$ identity matrix.
        Note that $R(q)$ and $L(p)$ commute, i.e. $R(q)L(p)=L(p)R(q)$.

Three-dimensional rotations in space can be represented by Euler parameters, i.e. unit
quaternions
\[
\mathbb{S}^3:=\{q \in \mathbb{H}\, | \, \|q\|=1 \}.
\]
Equipped with the quaternion product, $\mathbb{S}^3$ is a Lie group, with $q^{-1}=q^c$ and $e = (1,\mathbf{0})$ as the identity element.
There exists a (surjective $2:1$) group homomorphism (the Euler-Rodriguez map) $\mathcal{E}:\mathbb{S}^3 \to
SO(3)$, defined by
    \[
    \mathcal{E}(q) := I +2q_0\widehat{\mathbf{q}}+2\widehat{\mathbf{q}}^2.
    \]
The Euler-Rodriguez map can be explicitly written as
    \begin{equation*}
    \label{q_map_SO3}
    \mathcal{E}(q) = \left[
            \begin{array}{ccc}
            1-2(q_2^2+q_3^2)& 2(q_1q_2-q_0q_3)& 2(q_0q_2+q_1q_3)\\
            2(q_0q_3+q_1q_2)&1-2(q_1^2+q_3^2)&2(q_2q_3-q_0q_1)\\
            2(q_1q_3-q_0q_2)&2(q_0q_1+q_2q_3)&1-2(q_1^2+q_2^2)
            \end{array}
            \right].
    \end{equation*}
A rotation in
$\mathbb{R}^3$,
    \[
    \mathbf{w}=Q\mathbf{u},\quad Q\in SO(3),
    \quad \mathbf{u},\mathbf{w} \in \mathbb{R}^3,
    \]
    can, for some $q\in \mathbb{S}^3$, be expressed in quaternionic form as
    \begin{equation*}
    \label{rot_in_q}
  w=L(q)R(q^c) u=
    R(q^c)L(q) u,
    \quad u=(0,\mathbf{u}),\ w=(0,\mathbf{w})\in \mathbb{H}_{\mathcal{P}},
    \end{equation*}
where $\mathbb{H}_\mathcal{P}:=\{q\in \mathbb{H}\, | \,
q_0=0 \}\cong \mathbb{R}^3$ is the set of so called pure
quaternions.

\subsection{The Lie algebra $\mathfrak{s}^3$}
If $q\in \mathbb{S}^3$, it follows from
$qq^c=e$ that
\[
    \mathfrak{s}^3:=T_{e}\mathbb{S}^3= \mathbb{H}_{\mathcal{P}}.
\]
The Lie algebra $\mathfrak{s}^3$, associated with $\mathbb{S}^3$, is equipped
with a Lie bracket $[\, \cdot\, ,\cdot \,]_{\mathfrak{s}}:
\mathfrak{s}^3\times \mathfrak{s}^3 \to \mathfrak{s}^3$,
\[
    [\,u\, , w\,]_\mathfrak{s}:=
     L(u)w-L(w)u=(0,2 \mathbf{u}\times\mathbf{w}),
\]
where $u=(0,\mathbf{u}),\, w=(0,\mathbf{w})$.

The derivative map of $\mathcal{E}$ is 
$\mathcal{E}_{*} =T_{e}\mathcal{E}:\mathfrak{s}^3\to \mathfrak{so}(3)$. This map, given by
\begin{equation*}
\label{Map_s3_so3}
    \mathcal{E}_*(u)=2\widehat{\mathbf{u}},\quad u=(0,\mathbf{u}),
\end{equation*}
is a Lie algebra isomorphism.
Assume now that $q\in \mathbb{S}^3$
is such that
$\mathcal{E}(q(t))=Q(t)$, then
$L(q^c)\dot{q}\in \mathfrak{s}^3$, $Q^T\dot{Q}
\in \mathfrak{so}(3)$ and
\begin{equation}
\label{iso_dE_2}
  \mathcal{E}_{*}(L(q^{c})\dot{q}) = Q^T\dot{Q}. 
\end{equation}
Furthermore, it can be shown that
\begin{equation}
\label{eq_iso_2}
   \mathcal{E}_{*}(L(q)R(q^{c}) u ) =
   2 \widehat{\mathcal{E}(q)\mathbf{u}}
   \qquad \forall \, q\in\mathbb{S}^3 ,\, u=(0,\mathbf{u}) \in \mathfrak{s}^3.
\end{equation}
As a consequence of (\ref{iso_dE_2}) and (\ref{eq_iso_2}),
the kinematics of the attitude of the vessel \eqref{eq:attitudeSO3} can be expressed 
in Euler parameters in $\mathbb{S}^3$ as
\begin{equation}
\label{W_Om_in_q}
\dot{q} = \frac{1}{2}q\,\omega , \qquad \omega=(0,\boldsymbol{\omega}).
\end{equation}


\section{The equations using Euler parameters} \label{quatEq}

We rewrite the equations \eqref{eq:vessel}, \eqref{controls1}, \eqref{controls2}, \eqref{controls3}, using \eqref{W_Om_in_q} to represent the attitude with Euler parameters. Note that if  $q \in \mathbb{S}^3$ is known, we also know the Euler angles $\mbox{\boldmath $\theta$}$ from a transformation between
the two representations.

\begin{align*}
\label{eq:vessel}
\dot{\mathbf{p}} &= \mathbf{p}\times \mbox{\boldmath $\omega$} -\left(D_t\mathbf{v}+\mathcal{E}(q)^T\mathbf{g}_t^s-\mbox{\boldmath $\tau$}_t\right), \\
\dot{\mathbf{m}} &= \mathbf{m} \times \mbox{\boldmath $\omega$} + \mathbf{p}\times \mathbf{v}- \left(D_r\mbox{\boldmath $\omega$}+\mathcal{E}(q)^T\mathbf{g}_r^s-\mbox{\boldmath $\tau$}_r\right), \\
\dot{q} &= \frac{1}{2}q\,\omega , \qquad \omega=(0,\boldsymbol{\omega}),\\
\dot{\mathbf{x}} &= \mathcal{E}(q)\,\mathbf{v}, \\
\dot{\mbox{\boldmath $\varphi$}}_{\boldsymbol{\theta}}&=\tilde{\mbox{\boldmath $\theta$}}, \\
\dot{\mbox{\boldmath $\varphi$}}_{\mathbf{x}} &=\tilde{\mathbf{x}},
\end{align*}
with
\begin{align*}
\mbox{\boldmath $\tau$}_r  &= -\left(\Pi_e^{-1}\mathcal{E}(q)\right)^T(K_p^{r}\tilde{\mbox{\boldmath $\theta$}}
+ K_d^{r}\dot{\tilde{\mbox{\boldmath $\theta$}}} + K_i^{r} \mbox{\boldmath $\varphi$}_{\boldsymbol{\theta}}), \\
\boldsymbol{\tau}_t &= -\mathcal{E}(q)^T \left(K_p^{t}\tilde{\mathbf{x}}
+ K_d^{t}\dot{\tilde{\mathbf{x}}} + K_i^{t} \mbox{\boldmath $\varphi$}_{\mathbf{x}}\right),\\
\tilde{\mbox{\boldmath $\theta$}} &= \mbox{\boldmath $\theta$}- \mbox{\boldmath $\theta$}_{ref}, 
\hspace{6mm}
\dot{\tilde{\mbox{\boldmath $\theta$}}} = \dot{\mbox{\boldmath $\theta$}} = \Pi_e^{-1}\mathcal{E}(q)\, \mbox{\boldmath $\omega$} ,
\hspace{6mm}
\tilde{\mathbf{x}} = \mathbf{x}-\mathbf{x}_{ref}.
\end{align*}

\section{Splitting coefficients } \label{SplittingCoef}

The coefficients of a 4th order splitting scheme in the format (\ref{splittinggeneral}) are

{\scriptsize
\begin{center}
\begin{tabular}{lcl}
$a_1 = 0.0792036964311956500000000000000000000000$, && $b_1 = 0.209515106613361881525060713987$, \\
$a_2 = 0.353172906049773728818833445330$, && $b_2 = -0.14385177317981800000000000000000000$, \\
$a_3 = -0.042065080357719520000000000000000000000$, && $b_3 = \frac{1}{2}-(b_1+b_2)$,  \\
$a_4 = 1-2(a_1+a_2+a_3)$. && \\
\end{tabular}
\end{center}
}

The coefficients of a 6th order splitting scheme in the format (\ref{splittinggeneral}) are
{\small
\begin{center}
\begin{tabular}{lcl}
$a_1 = 0.0502627644003923808654389538920$, && $b_1 = 0.148816447901042828823498193483$, \\
$a_2 = 0.413514300428346618921141630839$, && $b_2 = -0.132385865767782744686048193902$, \\
$a_3 = 0.045079889794397660000000000000000000$, && $b_3 = 0.0673076046921849473963237618218$,\\
$a_4 = -0.188054853819571375656897886496$, && $b_4=0.432666402578172649872653897748$,\\
$a_5 = 0.541960678450781151905056284542$, &&$b_5=\frac{1}{2}-(b_1+b_2+b_3+b_4)$,\\
$a_6=1-2(a_1+a_2+a_3+a_4+a_5)$. && \\
\end{tabular}
\end{center}
}
We refer
to \cite{blanes} for an overview on splitting schemes. 

\section{Parameter values}\label{ParamValues}
The values of the parameters we use in the experiments are in SI units
\vspace{5mm}
\begin{center}
\begin{tabular}{lcl}
$D_{r1} = 9.329153987 \times 10^2$, &&  $D_{t1} = 3.53933789 \times 10^1$, \\
$D_{r2} =  6.514979127508227 \times 10^8$, && $D_{t2} = 1.1781388 \times 10^2$, \\
$D_{r3} =  3.15094664584 \times 10^4$,  && $D_{t3} = 1.4566249 \times 10^6$,  \\
$D_r = \mathrm{diag}(D_{r1},D_{r2},D_{r3})$, && $D_t = \mathrm{diag}(D_{t1},D_{t2},D_{t3})$, \\
$K_p^{r} = \mathrm{diag}(0,0,1\cdot 10^8)$, && $K_p^{t} = \mathrm{diag}(4\cdot 10^5,4\cdot 10^5,0)$,\\
$K_d^{r} = \mathrm{diag}(0,0,1\cdot 10^9)$, && $K_d^{t} = \mathrm{diag}(4\cdot 10^6,4\cdot 10^6,0)$,\\
$K_i^{r} = \mathrm{diag}(0,0,2\cdot 10^5)$,  && $K_i^{t} = \mathrm{diag}(1\cdot 10^3,1\cdot 10^3,0)$, \\
$T_1 = 2.873071 \times 10^8 $, && $\mbox{\boldmath $\theta$}_{0} = [0.05,-0.02,0.10]^T$, \\
$T_2 = 2.726143\times 10^9 $, &&  $\mbox{\boldmath $\theta$}_{ref} = \left[0,0,0.54\right]$,\\
$T_3 = 2.90000 \times 10^9 $, &&  $\mathbf{x}_0 = [723,0,0]^T$,  \\
$T =\mathrm{diag}(T_1,T_2,T_3)$, && $\mathbf{x}_{ref} = \left[780,20,0\right]$,\\
$\overline{GM}_T = 2.1440 $, && $m_v = 6.3622085\times 10^6$, \\
$\overline{GM}_L = 103.628 $, &&  $A_{wp} = 1.3834\times 10^3 $, \\
$g = 9.81 $, && 
\\
$\rho_w = 1.025 \times 10^3 $, && 
\\
$z_{eq} = 0$. && \\
\end{tabular}
\end{center}
Many of the values are taken from data for a supply vessel from \cite{mss}.